\documentclass[11pt]{amsart}

\usepackage{amsmath,amssymb,amsthm}
\usepackage[T1]{fontenc}
\usepackage{mathrsfs}
\usepackage{tikz-cd}
\usepackage[hidelinks]{hyperref}
\usepackage {oplotsymbl}

\newcommand{\var}{\operatorname{var}}
\newcommand{\can}{\operatorname{can}}

\newcommand{\M}[1]{\mathcal{M}_{0, #1}}
\newcommand{\Ms}[1]{\overline{\mathcal{M}}_{0, #1}}
\newcommand{\Q}{\mathbb Q}
\newcommand{\C}{\mathbb C}

\newcommand{\rev}{\operatorname{rev}}
\newcommand{\jjj}[1]{j_{\fontsize{5}{4}\selectfont\pentago \starletfill} #1}

\theoremstyle{plain}
\newtheorem{prop}{Proposition}
\newtheorem{theorem}{Theorem}
\newtheorem{lemma}{Lemma}

\newtheorem{cor}{Corollary}

\newtheorem*{theorem*}{Theorem}
\newtheorem*{cor*}{Corollary}

\theoremstyle{definition}
\newtheorem{definition}{Definition}

\theoremstyle{remark}
\newtheorem{rem}{Remark}

\begin{document}
\title{Pentagram relations}
\author{Nikita Markarian}

\begin{abstract}
We study the cohomology of iterated Kummer extensions on the
moduli space of five-pointed rational curves, continuing
\cite{MarkarianHodge,MarkarianDefect}.
A two-dimensional shrinking construction compares the Betti
presentations obtained from its five forgetful projections.
The corresponding Hodge defects give pentagram relations,
quadratic relations between coefficients of the Drinfeld
associator, which are equivalent to the pentagon equation
under group-likeness, normalization, and reversal.
Under reversal and vanishing of pure-power coefficients, the
relations imply shuffle and regularized double shuffle,
whose Hodge realization is recovered by a convolution specialization.
\end{abstract}

\email{nikita.markarian@gmail.com}

\date{}

\address{UMR 7501, Université de Strasbourg,
7 rue René Descartes,
67084 Strasbourg Cedex, France}
\maketitle

\section*{Introduction}

This paper continues the study of Hodge comparison maps
begun in \cite{MarkarianHodge,MarkarianDefect}. In those
papers, the comparison between topological descriptions
of cohomology and Hodge filtrations is expressed through
the coefficients of the Drinfeld associator. Here we
study the compatibility of these comparisons when the
same cohomology is calculated by different projections.

In \cite{MarkarianDefect}, regularized double shuffle
relations arise from a Hodge comparison associated with
multiplicative convolution. We generalize this approach
by considering an arbitrary iterated Kummer variation on
$\M{5}$ and the five forgetful projections to $\M{4}$.
We extend this variation
to the compactification by $*$ and $!$ along boundary
divisors, according to a fixed cyclic order. Each projection
gives a calculation of the resulting second cohomology.
The key geometric fact is
that spreading a corner value over the real pentagon
gives the same Betti class for either adjacent projection.
This provides the common Betti maps needed to compare
the Hodge calculations.

The resulting commutative diagram
\eqref{eq:pentagram-diagram} gives relations between
associator coefficients, which we call the
\emph{pentagram relations}. We prove that these relations
are equivalent to the pentagon equation for normalized
group-like series satisfying reversal
(Theorem~\ref{thm:pentagon-pentagram}).
Group-likeness itself follows from the pentagram equations
under reversal and the normalization that all pure-power
coefficients vanish (Subsection~\ref{subsec:shuffle}).

The pentagram formulation expresses compatibility of
Hodge comparison maps by equations of degree at most two
in associator coefficients. For the period associator,
these are quadratic relations between multiple zeta values.
They admit an explicit presentation as the Maurer--Cartan
equation of a complete differential graded Lie algebra
(Appendix~\ref{app:maurer-cartan}).
The implication from pentagram to pentagon, together with
\cite[Theorem~1.1]{FurushoDoubleShuffle}, gives regularized
double shuffle for every such formal series. For the period
associator, the convolution specialization recovers the Hodge
comparison of \cite{MarkarianDefect}.

In \cite{MarkarianHodge} we suggested interpreting
associator coefficients as higher terms of a cohomological
functor on an $A_\infty$ category of Hodge variations on
$\M{4}$. From this perspective, pushforwards from $\M{5}$
and their compatibility with pushforward to a point should
account for relations between these coefficients.
The present construction gives a cohomological instance
of this compatibility. The equivalence with pentagon
gives a concrete illustration of the philosophy of
Grothendieck's \emph{Esquisse d'un programme}
\cite{GrothendieckEsquisse}: the Hodge comparison relations
studied here are governed by the pentagon relation in
the genus-zero Teichm\"uller tower.
We also expect an analogous
construction for iterated extensions of Tate motives.
The Maurer--Cartan presentation raises the question of its
relation to the motivic Galois group of mixed Tate motives
over $\mathbb Z$.

Sections~\ref{sec:cohomology-functor}--\ref{sec:betti-presentations}
construct the common Betti presentation,
Section~\ref{sec:pentagram-relations} derives the relations,
and Section~\ref{sec:consequences-specializations} treats
pentagon, convolution, and shuffle. Appendix~\ref{app:betti-comparison}
calculates the second column of the Betti comparison;
Appendix~\ref{app:logarithmic-calculations} gives the logarithmic
and frame calculations, and
Appendix~\ref{app:coefficient-calculations} gives the Maurer--Cartan
presentation and the coefficient calculations through weight five.

\medskip
\noindent\textit{Acknowledgements.}
I would like to thank IHES for hospitality and excellent
working conditions.
This project was supported by the PAUSE program and the ITI IRMIA++.

\section{Mixed extensions and forgetful projections}
\label{sec:cohomology-functor}

\subsection{Iterated Kummer extensions}
\label{subsec:iterated-kummer}

Put
\[
U=\M{5},\qquad S=\Ms{5}.
\]
\begin{definition}
\label{def:iterated-kummer}
Following \cite[Section~1]{MarkarianHodge}, consider
admissible unipotent variations $\mathbb V$ of mixed
Hodge--Tate structures on $U$ with constant Tate graded
pieces. Our weight convention is
\[
\operatorname{gr}_{-2n}^W\mathbb V
 \simeq\mathbb Q_U(n)^{\oplus r_n},
\qquad
\operatorname{gr}_{2n+1}^W\mathbb V=0.
\]
Write $V$ for the underlying rational local system.

The basic Kummer extensions considered in this paper are
attached to rational functions $g$ on $S$ whose zeros and
poles are simple and lie in $S\setminus U$:
\[
0\longrightarrow\mathbb Q_U(1)
 \longrightarrow K_g
 \longrightarrow\mathbb Q_U(0)
 \longrightarrow0.
\]
We consider finite iterated extensions with these Kummer classes
and their Tate twists as logarithmic connection data: in the
Hodge grading,
\[
 E=\bigoplus_n E_n,\qquad
 \nabla=d-\sum_\alpha N_\alpha\,d\log g_\alpha,
 \qquad N_\alpha E_n\subset E_{n+1},\qquad \nabla^2=0,
\]
where $E_n$ has Tate type $\Q(n)$ and the functions $g_\alpha$
are as above. The admissible rational Betti realization and
its comparison with this filtered flat bundle are part of the data.
\end{definition}

\begin{rem}
The de Rham class of $K_g$ is $d\log g$. On
$\M{4}=\mathbb P^1\setminus\{0,1,\infty\}$ we retain
the convention of \cite{MarkarianHodge}: the two basic
classes are $dz/z$ and $dz/(1-z)$, so the latter is the
negative of the standard Kummer class of $1-z$.
\end{rem}

Pullback along a forgetful map and tensor product preserve this
class, which therefore contains the tensor products of pullbacks
used in \cite{MarkarianDefect}. The Betti constructions apply to any
finite-rank unipotent rational local system $V$; the Hodge
calculations retain the variation $\mathbb V$.

\subsection{Shrinking}
\label{subsec:shrinking}

All sheaf-theoretic functors are derived. As in
\cite{MarkarianConvolution}, nearby and vanishing cycles
are denoted by $\psi_f$ and $\phi_f$.
Their perverse normalizations are $\psi_f[-1]$ and
$\phi_f[-1]$. When the support is a point, put
\[
\Psi_f(\mathcal F)
 =H^0\bigl(\psi_f(\mathcal F)[-1]\bigr),
\qquad
\Phi_f(\mathcal F)
 =H^0\bigl(\phi_f(\mathcal F)[-1]\bigr).
\]
We use the classical Betti variation
\[
\var:\Phi_f(\mathcal F)
   \longrightarrow\Psi_f(\mathcal F),
\qquad
\var\circ\can=T_f-1.
\]
The logarithmic variation in the Hodge realization is
denoted by $\operatorname{Var}$.

Let
\[
j_0:\mathbb C^*\hookrightarrow\mathbb A^1,
\qquad
j_\infty:\mathbb A^1\hookrightarrow\mathbb P^1
\]
be the inclusions.

\begin{prop}[Shrinking]
Let $\mathcal M$ be a perverse sheaf on $\mathbb C^*$
smooth outside $1$. The interval $I=(0,1]$ gives a natural isomorphism
\[
\varphi_I:\Phi_{1-t}(\mathcal M)
   \xrightarrow{\sim}H^0\bigl(\mathbb P^1,j_{\infty!}j_{0*}\mathcal M\bigr).
\]
Moreover,
\[
H^k\bigl(\mathbb P^1,j_{\infty!}j_{0*}\mathcal M\bigr)=0
\qquad(k\ne0).
\]
Both functors are exact.
\end{prop}

\begin{proof}
The supported-cohomology description of vanishing cycles
identifies
\[
\Phi_{1-t}(\mathcal M)
 =H^0_{(0,1]}(\mathbb C^*,\mathcal M).
\]
Forgetting support gives the shrinking isomorphism of
\cite[Section~1]{MarkarianConvolution}. Cohomological
concentration and the long exact sequence give exactness.
\end{proof}

We use the same notation on any three-pointed projective line
with a chosen coordinate. Subscripts record the extension
conditions at $0,1,\infty$: thus $j_{**!}$ means $*$ at
$0,1$ and $!$ at $\infty$, and similarly for $j_{*!!}$.

\begin{rem}
For the Hodge--Tate objects considered in
\cite{MarkarianHodge,MarkarianDefect}, the spaces on both
sides carry mixed Hodge structures and $\varphi_I$ is
compatible with weights. Its comparison with the Hodge
realization is the Hodge defect of shrinking studied in these papers.
\end{rem}

\subsection{The mixed extension and its pushforward}
\label{subsec:mixed-extension}

The boundary divisors are indexed by unordered pairs of
markings. Fix
\[
\begin{aligned}
D_*&=D_{12}\cup D_{23}\cup D_{34}\cup D_{45}\cup D_{15},\\
D_!&=D_{13}\cup D_{24}\cup D_{35}\cup D_{14}\cup D_{25}.
\end{aligned}
\]
Let $V$ be a finite-rank unipotent rational local system on $U$.

\begin{definition}
\label{def:mixed-extension}
Consider the inclusions
\[
U\xrightarrow{\,a\,}S\setminus D_*
 \xrightarrow{\,b\,}S
\]
and put
\[
\jjj{V}=b_*a_!V.
\]
This is the mixed extension with $*$ along $D_*$ and $!$
along $D_!$. For a variation $\mathbb V$, its cohomology
$H^2(S,\jjj{V})$ underlies the mixed Hodge structure
$H^2(S,\jjj{\mathbb V})$.
\end{definition}

\begin{lemma}
\label{lem:mixed-extension-locality}
The complex $\jjj{V}[2]$ is perverse. Its ordinary
restriction to every $!$-component is zero, including at its
intersections with $*$-components. Verdier duality
exchanges the two boundary conditions.
\end{lemma}

\begin{proof}
At a normal crossing, the derived $*$- and $!$-extensions
in the two independent normal directions commute. This
gives the restriction and duality assertions. The open
embeddings $a$ and $b$ are affine, so $a_!$ and $b_*$
preserve perversity.
\end{proof}

Let
\[
p_i:S\longrightarrow\Ms{4}
\]
forget the $i$-th marking and stabilize. We first consider
$p=p_3$. Write the five markings as
\[
(0,\infty,1,x,y)
\]
and use the target coordinate
\[
t=1-y/x.
\]
The boundary fibers and their extension conditions are
\[
\begin{array}{c|c|c}
t&\text{fiber}&\text{conditions}\\ \hline
0&D_{12}\cup D_{45}&*,*\\
1&D_{15}\cup D_{24}&*,!\\
\infty&D_{14}\cup D_{25}&!,!
\end{array}
\]

Put
\[
\mathcal F=p_*\jjj{V}[2].
\]

\begin{prop}[The forgetful pushforward and shrinking]
\label{prop:forgetful-pushforward}
The complex $\mathcal F$ is perverse. It is the $*$-extension
at $0$ and the $!$-extension at $\infty$ of its restriction
to $\mathbb C^*$:
\[
\mathcal F\simeq
j_{\infty!}j_{0*}(\mathcal F|_{\mathbb C^*}).
\]
Its restriction to $\M{4}$ is a local system of rank
$2\operatorname{rank}V$, shifted by $[1]$. Consequently,
shrinking gives a natural isomorphism
\[
\varphi_I:\Phi_{1-t}(\mathcal F)
 \xrightarrow{\sim}H^0(\mathbb P^1,\mathcal F)
 \xrightarrow{\sim}H^2(S,\jjj{V}).
\]
\end{prop}

\begin{proof}
The two opens $S\setminus D_*$ and $S\setminus D_!$
are Brown's affine dihedral partial compactifications
\cite[Lemmas~2.30--2.31]{BrownModuli}. Hence the map
\[
S\setminus D_*\longrightarrow\mathbb P^1\setminus\{0\}
\]
is affine. Its direct image, followed by $*$-extension
at $0$, gives the upper perverse bound for $\mathcal F$.
Duality and the affine map
\[
S\setminus D_!\longrightarrow\mathbb P^1\setminus\{\infty\}
\]
give the lower bound. The same factorizations give the
two extension conditions on the base.

A smooth fiber is a four-pointed projective line with
two $*$-points and two $!$-points. Its cohomology is
concentrated in degree one, of dimension
$2\operatorname{rank}V$, as the relative cocycle calculation in
Subsection~\ref{subsec:projection-vanishing-cycles} shows.
The last assertion follows by shrinking and proper direct image.
\end{proof}

Cyclic relabelling gives the corresponding construction
for the other forgetful projections.

\section{Betti presentations}
\label{sec:betti-presentations}

\subsection{Vanishing cycles of a forgetful projection}
\label{subsec:projection-vanishing-cycles}

\begin{definition}[Component nearby and vanishing cycles]
For $p=p_3$, put
\[
f=1-t,\qquad
Z=p^{-1}(1)=C_*\cup C_!,
\]
where
\[
C_*=D_{15},\qquad C_!=D_{24},\qquad
o=C_*\cap C_!.
\]
Let $h:Z\hookrightarrow S$ be the inclusion. Writing also
$f=(1-t)\circ p$ on $S$, put
\[
\mathscr N=\psi_f(\jjj{V}[2])[-1],
\qquad
\mathscr P=\phi_f(\jjj{V}[2])[-1].
\]
Proper direct image gives
\[
\begin{aligned}
\Psi&:=\Psi_{1-t}(\mathcal F)=H^0(Z,\mathscr N),\\
\Phi&:=\Phi_{1-t}(\mathcal F)=H^0(Z,\mathscr P).
\end{aligned}
\]

Let $L_*$ and $L_!$ be the normal nearby-cycle local
systems on the open components, defined by
\[
\mathscr N|_{C_\bullet^\circ}=L_\bullet[1].
\]
Here $C_\bullet^\circ$ is obtained by removing its three
boundary points. On $C_*$ place the external points at
$0,1$ and the node at $\infty$. On $C_!$ place the node
at $0$ and the external points at $1,\infty$. Put
\[
E=H^1(C_*,j_{**!}L_*),
\qquad
Q=H^1(C_!,j_{*!!}L_!).
\]
Both spaces have dimension $\operatorname{rank}V$.

Write $T_*$ for the normal monodromy about $D_{15}$,
acting on $L_*$ and $E$, and put
\[
\delta_*=T_*-1.
\]

\end{definition}

\begin{prop}
There is a commutative diagram of exact sequences
\[
\begin{tikzcd}[column sep=small,row sep=large]
0\arrow[r]
&E\arrow[r,"a"]\arrow[d,"\delta_*"']
&\Psi\arrow[r,"\rho"]\arrow[d,"\can"]
&Q\arrow[r]\arrow[d,equal]
&0\\
0\arrow[r]
&E\arrow[r,"\iota"']
&\Phi\arrow[r,"\pi"']
&Q\arrow[r]
&0.
\end{tikzcd}
\]
The map $a$ is represented by the interval joining the
two $*$-points in a nearby smooth fiber.
\end{prop}

\begin{proof}
Let
\[
i:C_!\hookrightarrow Z,
\qquad
k:C_*\setminus\{o\}\hookrightarrow Z
\]
be the inclusions. The original mixed extension satisfies
\[
(hi)^*\jjj{V}=0.
\]
The Milnor triangle therefore gives an isomorphism
\[
i^*(\can):
i^*\mathscr N\xrightarrow{\sim}i^*\mathscr P.
\]

At the node write $f=uv$, with $C_*=(u=0)$ and
$C_!=(v=0)$. Projection to $u$ identifies the cohomology
of the nearby annulus with that of a punctured neighborhood
of the node on $C_!$. Thus $i^*\mathscr N$ extends
$L_![1]$ by $*$ at the node and by $!$ at the external
points, giving
\[
H^0(C_!,i^*\mathscr N)
 =H^0(C_!,i^*\mathscr P)=Q.
\]

On the open $*$-component, classical variation identifies
$\mathscr P$ with $\mathscr N$. After this identification,
$\can$ is $T_*-1$. The complexes
$k_!k^*\mathscr N$ and $k_!k^*\mathscr P$ are therefore
computed by the $j_{**!}$-extension on $C_*$.

The localization triangles now give the diagram: shrinking
concentrates the component terms in degree zero, with
cohomology $E$ and $Q$.

The interval on $C_*$ avoids the node. Specializing in a
product normal neighborhood gives the interval between
the two nearby $*$-points; naturality of support maps
identifies its class with $a$.
\end{proof}

\begin{rem}
In the Hodge realization, the left term of the lower
sequence is
\[
H^1(C_*,j_{**!}\mathbb L_*)(-1).
\]
The diagram above uses the classical Betti variation
to identify its underlying vector space with $E$.
The conversion to logarithmic variation is recorded in
Subsection~\ref{subsec:one-projection}.
\end{rem}

\begin{cor}
\label{cor:vanishing-cycle-pushout}
The vanishing-cycle space has the presentation
\[
\Phi\simeq
H^0\left[
E\xrightarrow{\ (a,-\delta_*)\ }\Psi\oplus E
\right],
\]
where the two terms have degrees $-1,0$. Explicitly,
\[
\Phi=
\frac{\Psi\oplus E}
{\{(a(v),-\delta_*v):v\in E\}},
\]
and the quotient map sends $(\xi,e)$ to
$\can(\xi)+\iota(e)$.
\end{cor}

\begin{proof}
The preceding diagram identifies the lower extension
with the pushout of the upper one by $\delta_*$. The
map $a$ is injective, so the displayed complex has no
cohomology in degree $-1$.
\end{proof}

We use the interval presentation from the proof of
\cite[Theorem~3]{MarkarianConvolution}, retaining the relative
terms at both $!$-points. A nearby real fiber $0<t_0<1$
has coordinate $x$ and marked points
\[
0,\quad 1,\quad\lambda=(1-t_0)^{-1},\quad\infty,
\]
with conditions $*,*,!,!$.
Choose a base point, paths to the punctures, and positive
meridians satisfying
\[
g_0g_1g_\lambda g_\infty=1.
\]
Let $V_f$ be the coefficient fiber and $T_j$ the corresponding
left monodromy operators. Put
\[
R=T_0T_1T_\lambda=T_\infty^{-1}.
\]

Fix the paths to $\lambda$ and $\infty$ disjoint from the
short interval joining $0$ to $1$. Fix its normal orientation
and identify its coefficient sections with $V_f$ so that
the class of $v$ has value $+v$ on $g_1$. Use the same
coefficient paths on $C_!$.

\begin{lemma}
\label{lem:relative-fiber-coordinates}
The relative cocycles
\[
c(z,q)=\bigl((R-1)q-T_0z,\ z,\ 0;\ 0,\ q\bigr),
\qquad z,q\in V_f,
\]
give an isomorphism
\[
\Psi\xrightarrow{\sim}V_f\oplus V_f,
\qquad [c(z,q)]\longmapsto(z,q).
\]
The interval and relative-loop identifications
$E=V_f$ and $Q=V_f$ satisfy
\[
a(v)=[c(v,0)],\qquad \rho(z,q)=q.
\]
In particular,
\[
\ell:Q\longrightarrow\Psi,
\qquad \ell(q)=[c(0,q)]
\]
is a section of $\rho$.
\end{lemma}

\begin{proof}
A degree-one relative cocycle is represented by
\[
(z_0,z_1,z_\lambda;w_\lambda,w_\infty)\in V_f^{\oplus5}.
\]
The first three entries are crossed-homomorphism values
on the meridians; the last two are degree-zero terms
in the peripheral circle complexes at the $!$-points.
The cocycle equations are
\[
\begin{aligned}
z_\lambda&=(T_\lambda-1)w_\lambda,\\
z_0+T_0z_1+T_0T_1z_\lambda&=(R-1)w_\infty.
\end{aligned}
\]
Coboundaries are
\[
\bigl((T_0-1)v,(T_1-1)v,(T_\lambda-1)v;v,v\bigr),
\qquad v\in V_f.
\]
They are obtained from the mapping fiber of restriction
from the punctured sphere to its two peripheral circle
complexes.

Every class has a unique representative with
$w_\lambda=0$. Solving the two equations then gives
$c(z,q)$, proving the first assertion.

The interval class has value $+v$ on $g_1$ and zero on
$g_\lambda,g_\infty$. The meridian relation forces the
value $-T_0v$ on $g_0$, giving $a(v)=[c(v,0)]$.

Restriction of $\mathscr N$ to $C_!$ forgets the two
meridians on the other component. With the retained
coefficient paths, its class is determined by
$q=w_\infty-w_\lambda$. Hence $\rho(z,q)=q$ and
$\ell(q)=[c(0,q)]$ is a section; its term $(R-1)q$
is the monodromy correction to the relative loop.

\end{proof}

\begin{prop}
With these paths and coefficient identifications, the map
\[
\vartheta_\ell:E\oplus Q\longrightarrow\Phi,
\qquad
(e,q)\longmapsto
\iota(e)+\can\ell(q)
\]
is an isomorphism of Betti spaces.
\end{prop}

\begin{proof}
By Corollary~\ref{cor:vanishing-cycle-pushout}, the quotient
coordinates on $\Phi$ are
\[
[(z,q,e)]\longmapsto(e+\delta_*z,q).
\]
Under this identification,
\[
\iota(e)=(e,0),\qquad
\pi(e,q)=q,\qquad
\can(z,q)=(\delta_*z,q).
\]
In particular,
$\can\ell(q)=(0,q)$, proving the assertion.
\end{proof}

\begin{rem}
A basis of $V_f$ gives the interval and lifted-loop basis
of $\Phi$, hence of $H^2(S,\jjj{V})$ by $\varphi_I$.
This presentation depends on the paths: replacing $\ell$
by $\ell+ah$, for $h:Q\to E$, changes the lift by
\[
\can(\ell+ah)
 =\can\ell+\iota\delta_*h.
\]
\end{rem}

\begin{cor}
\label{cor:global-concentration}
For every unipotent local system $V$ on $\M{5}$,
\[
H^i(S,\jjj{V})=0\qquad(i\ne2),
\qquad
\dim H^2(S,\jjj{V})=2\operatorname{rank}V.
\]
\end{cor}

\begin{proof}
By Proposition~\ref{prop:forgetful-pushforward}, shrinking
identifies the only cohomology group with $\Phi$, of
dimension $\dim E+\dim Q$.
\end{proof}

\begin{rem}
The vanishing also follows directly from Artin vanishing
on the two affine opens used in
Subsection~\ref{subsec:mixed-extension}. Their dihedral
orders are
\[
S\setminus D_!=\M{5}^{\delta},
\qquad
S\setminus D_*=\M{5}^{\delta'},
\]
where $\delta=(12345)$ and $\delta'=(13524)$.
Artin vanishing on $S\setminus D_*$, applied to the
perverse sheaf $a_!V[2]$, gives vanishing in degrees
greater than $2$. Verdier duality exchanges the two
extension conditions, so the same argument on
$S\setminus D_!$ gives vanishing in degrees less than $2$.
\end{rem}

\subsection{Two-dimensional shrinking}
\label{subsec:two-dimensional-shrinking}

Put $Y=S\setminus D_*$. We construct a presentation by
sections supported on two disks. The first is the main real
pentagon. The second is obtained by cutting the remaining real
chambers along $D_*$ and moving their $!$-edges in imaginary
normal directions.

Write the markings as $(0,\infty,1,x,y)$. Let $P$ and $P^{\mathrm{opp}}$
be the closures in $S(\mathbb R)$ of the chambers
\[
P^\circ=\{0<y<x<1\},\qquad
(P^{\mathrm{opp}})^\circ=\{x<0,\ 1<y\}.
\]
The sides of $P$ lie in $D_*$, and those of $P^{\mathrm{opp}}$ lie in
$D_!$. Moreover,
\[
P\cap D_!=\varnothing,\qquad P^{\mathrm{opp}}\cap D_*=\varnothing.
\]
We orient the real chambers by
\[
\eta=\frac{dx\wedge dy}{y(x-y)(1-x)}.
\]
This form has divisor $-D_*$ and is therefore nonvanishing
on $Y(\mathbb R)$.

Cut $S(\mathbb R)$ along all the curves $D_*(\mathbb R)$,
separating the quadrants at their intersections.
One component of the resulting surface is $P$.
Denote the remaining cut surface by $\Sigma$ and its natural
map to $S(\mathbb R)$ by $r$. The following lemma shows that
$\Sigma$ is connected.

\begin{lemma}
\label{lem:eleven-pentagon-disk}
The surface $\Sigma$ is a disk consisting of the eleven
pentagons other than $P$. The inverse image of each
component of $D_!(\mathbb R)$ is a properly embedded
interval containing one side of $P^{\mathrm{opp}}$.
\end{lemma}

\begin{proof}
Label a chamber by its cyclic order, modulo reversal and
cyclic rotation. Crossing $D_{ij}$ interchanges the adjacent
labels $i,j$. At a vertex, the two incident pairs are disjoint,
so these interchanges commute. The following table records
the attachments around $P^{\mathrm{opp}}$, of order $13524$.
The sides in the first column occur in boundary order; the
last column gives the chamber opposite $P^{\mathrm{opp}}$
at the vertex shared with the next side, cyclically.
\[
\begin{array}{c|c|c}
\text{side}&\text{chamber across the side}
 &\text{chamber at the next vertex}\\ \hline
D_{13}&13425&12543\\
D_{25}&13254&14325\\
D_{14}&12534&12354\\
D_{35}&14235&12435\\
D_{24}&12453&13245
\end{array}
\]
Start with $P^{\mathrm{opp}}$ and attach the five chambers
in the middle column along their indicated sides. At each
vertex of $P^{\mathrm{opp}}$, three of the four incident
chambers have then been included. The chamber in the last
column attaches along the two exposed edges meeting at that
vertex. Its other three sides lie in $D_*$ and remain cut.
The table lists ten distinct chambers, all different from
$P$ and $P^{\mathrm{opp}}$, so it accounts for the entire
remaining surface. Every attachment is along a boundary
interval, and there are no further uncut edges to identify.
Thus $\Sigma$ is a disk.

Each component of $D_!(\mathbb R)$ is a circle meeting $D_*$
at exactly one point. Cutting at that point gives an interval;
the table places its three consecutive edges in $\Sigma$,
with the middle edge a side of $P^{\mathrm{opp}}$.
Its endpoints lie on $\partial\Sigma$.
\end{proof}

Choose the normal direction to each of these five
intervals to point outward from $P^{\mathrm{opp}}$ along
its side, and extend it along the interval. Put
\[
X=S\setminus D_!,\qquad D=D_*\cap X.
\]

\begin{lemma}
\label{lem:displaced-disk}
For sufficiently small $\epsilon>0$, there is an imaginary
normal displacement
\[
F_\epsilon:(\Sigma,\partial\Sigma)\longrightarrow(X,D)
\]
with the prescribed normal signs, such that
\[
F_\epsilon^{-1}(D)=\partial\Sigma,
\qquad F_\epsilon(\Sigma^\circ)\subset U.
\]
Its interior is embedded and closed in $Y$, and is disjoint
from $P^\circ$. The displacement is unique up to a homotopy
through such normal graphs. It may be chosen to intersect
$P^{\mathrm{opp}}$ at the single nondegenerate zero of an
outward-pointing vector field on $P^{\mathrm{opp}}$.
\end{lemma}

\begin{proof}
Choose a real vector field $v$ along $r$ with the following
properties:
its normal component along each of these five intervals
has the chosen sign; it is tangent to $D_*$ on
$\partial\Sigma$; and its restriction to $P^{\mathrm{opp}}$ points
outward along $\partial P^{\mathrm{opp}}$ and has one nondegenerate
zero of index $1$ in $(P^{\mathrm{opp}})^\circ$.
Such a field exists. The normal directions extend
along the intervals, and at their crossings the two
normal conditions are independent. At an intersection
of $D_!$ with $D_*$, the prescribed transverse direction
is compatible with tangency to $D_*$. The resulting
local fields can be patched on the cut surface.

Take the imaginary normal graph in the direction $v$,
using tubular coordinates compatible with the boundary
divisors. Its tubular projection to $S(\mathbb R)$ is $r$.
At a crossing of two components of $D_!$,
the local construction has the form
\[
(u,w)\longmapsto
\bigl(u+i\epsilon v_u,\ w+i\epsilon v_w\bigr),
\]
with both normal components nonzero on the corresponding
real axes. At a crossing with $D_*$ the deformation
preserves $D_*$. This construction is made on $\Sigma$;
different boundary points of $\Sigma$ mapping to the same
point of $S(\mathbb R)$ retain their separate branches.

Interpolation in the prescribed normal half-spaces,
preserving the boundary in $D_*$, joins any two sufficiently
small choices by a homotopy of normal graphs.

Since $r|_{\Sigma^\circ}$ is injective, the normal graph
$F_\epsilon|_{\Sigma^\circ}$ is an embedding. Its image and
$P^\circ$ have disjoint real projections and are closed
in $Y$: their boundaries lie in the removed divisor
$D_*$.
\end{proof}

\begin{definition}
\label{def:disk-supported-maps}
Write
\[
i_P:P^\circ\hookrightarrow Y,\qquad
i_\Sigma:\Sigma^\circ\hookrightarrow Y
\]
for these embeddings, with $i_\Sigma=F_\epsilon|_{\Sigma^\circ}$,
and put
\[
S_P(V)=\Gamma(P^\circ,V),\qquad
S_\Sigma(V)=\Gamma(\Sigma^\circ,F_\epsilon^*V).
\]
Both spaces have dimension $\operatorname{rank}V$.

Set $\mathcal G(V)=a_!V[2]$.
The disk and complex orientations trivialize the normal
orientation lines. Both images lie in $U$, so real purity
in codimension two gives
\[
i_P^!\mathcal G(V)\simeq V|_{P^\circ},
\qquad
i_\Sigma^!\mathcal G(V)\simeq F_\epsilon^*V.
\]
The adjunction morphisms
\[
i_{\alpha*}i_\alpha^!\mathcal G(V)\longrightarrow
\mathcal G(V),\qquad \alpha=P,\Sigma,
\]
therefore induce maps
\[
\varphi_P:S_P(V)\longrightarrow H^2(S,\jjj{V}),\qquad
\varphi_\Sigma:S_\Sigma(V)\longrightarrow H^2(S,\jjj{V}).
\]
We use the fundamental classes with the orientations
specified above.
\end{definition}

\begin{rem}
Intrinsically, the source of each support map consists of
coefficient sections tensored with the normal orientation
line. Relabelling transports this line as well: a reflection
reverses the normal orientation of the real pentagon and
contributes the sign $-1$ to the coefficient identification.

Transport along the homotopy of Lemma~\ref{lem:displaced-disk}
makes the maps independent of the sufficiently small
displacement with the prescribed signs.
\end{rem}

\begin{prop}[Two-dimensional shrinking]
\label{prop:two-dimensional-shrinking}
The map
\[
\varphi_{P,\Sigma}=\varphi_P+\varphi_\Sigma:
S_P(V)\oplus S_\Sigma(V)\xrightarrow{\sim}H^2(S,\jjj{V})
\]
is an isomorphism, natural in $V$.
\end{prop}

\begin{proof}
First take $V=\mathbb Q_U$. Using the complex orientation
and exchanging the two extension conditions, Poincar\'e--Verdier
duality gives
\[
H^2(S,\jjj{\mathbb Q_U})\simeq H_2(X,D;\mathbb Q).
\]
The two supported classes become
\[
[P],\qquad
z_\Sigma=F_{\epsilon*}[\Sigma,\partial\Sigma].
\]

Pair with $[P^{\mathrm{opp}}]\in H_2(Y,D_!\cap Y;\mathbb Q)$.
Disjointness gives
\[
[P]\cdot[P^{\mathrm{opp}}]=0.
\]
The displaced disk meets $P^{\mathrm{opp}}$ at the unique
zero of $v|_{P^{\mathrm{opp}}}$, of index $1$.
In complex dimension two the intersection sign of an
imaginary normal graph is the negative of this index,
so the orientations induced by $\eta$ give
\[
z_\Sigma\cdot[P^{\mathrm{opp}}]=-1.
\]
Each component of $D$ is an affine line, obtained by
removing its unique intersection with $D_!$, and their
incidence graph is a pentagon. Hence
$H_1(D;\mathbb Q)=\mathbb Q$, generated by $\partial[P]$.
In particular, $[P]\ne0$, so $[P]$ and $z_\Sigma$
are independent. Since $H_2(X,D;\mathbb Q)$ has
dimension two, they form a basis.

By Corollary~\ref{cor:global-concentration}, the functor
$V\mapsto H^2(S,\jjj{V})$ is exact on unipotent local systems. The functors $S_P$ and
$S_\Sigma$ are exact because the interiors of the
two disks are contractible. Naturality of the
adjunction maps and induction on a unipotent filtration
now prove the proposition.
\end{proof}

Put
\[
L_P(V)=\operatorname{im}\varphi_P,\qquad
L_\Sigma(V)=\operatorname{im}\varphi_\Sigma.
\]
The proposition gives the Betti decomposition
\[
H^2(S,\jjj{V})=L_P(V)\oplus L_\Sigma(V).
\]

\subsection{Compatibility with successive shrinking}
\label{sec:betti-comparison}

We compare $\varphi_{P,\Sigma}$ with the Betti presentation
$\varphi_I\vartheta_\ell:E\oplus Q\xrightarrow{\sim}H^2(S,\jjj{V})$
of Subsection~\ref{subsec:projection-vanishing-cycles}.
The main-pentagon compatibility below is the part used to
prove the geometric pentagram relations. The second column
completes the comparison of the two Betti presentations;
its calculation is given separately in
Appendix~\ref{app:betti-comparison}.

For a dihedral order $\tau$, denote by $P_\tau$ the
closure of its real chamber; thus $P^{\mathrm{opp}}=P_{13524}$.
Choose $q_c\in F_\epsilon((P^{\mathrm{opp}})^\circ)$
and put $V_b=V_{q_c}$. Evaluation at $q_c$
identifies the flat sections on $\Sigma^\circ$ with $V_b$.
Transports within the displaced disk are path-independent.

Retain the nearby-fiber coefficient space $V_f$ and the
identifications
\[
E=V_f,\qquad Q=V_f.
\]
\begin{lemma}
\label{lem:main-pentagon-compatibility}
The component shrinking map
$\alpha:S_P(V)\xrightarrow{\sim}E$ satisfies
\[
\varphi_P=\varphi_I\iota\alpha.
\]
\end{lemma}

\begin{proof}
The coordinates $y=(1-t)x$ identify $P^\circ$ with
$(0,1)_t\times(0,1)_x$, and
\[
\eta=\frac{dt\wedge dx}{t(1-t)x(1-x)}.
\]
The support map of this product is the composition of
the interval support maps, with these normal orientations.
At $t=1$ the $x$-interval lies on $C_*$ and avoids the node;
its specialization is $\iota\alpha$. Outer shrinking
therefore gives the equality.
\end{proof}

For the second column, use $S_\Sigma(V)=V_b$ and the
coefficient paths of Appendix~\ref{app:betti-comparison}. Let
\[
h_I,h_A,h_B,h_Q:V_b\longrightarrow V_f
\]
be its four parallel transports, viewed as maps to $E$
for the first three and to $Q$ for the fourth. The
operators $T_1$ (monodromy about $x=1$) and $T_*$ act
through the same identification with $V_f$.

\begin{prop}
\label{prop:comparison-betti-presentations}
With the stated coefficient paths and orientations,
\[
\vartheta_\ell^{-1}\varphi_I^{-1}
\varphi_{P,\Sigma}
=
\begin{pmatrix}
\alpha&K_\Sigma\\
0&h_Q
\end{pmatrix},
\]
where
\[
K_\Sigma
=
h_I+h_A+T_*h_B+T_*(T_1-1)h_Q.
\]
\end{prop}

\begin{proof}
The first column is Lemma~\ref{lem:main-pentagon-compatibility}.
The second-column calculation is given in
Appendix~\ref{app:betti-comparison}.
\end{proof}

\begin{cor}
For constant coefficients the comparison matrix is
\[
\begin{pmatrix}
1&3\\
0&1
\end{pmatrix}.
\]
The second column has boundary three times the boundary
of the main pentagon.
\end{cor}

\begin{proof}
All four transports and both monodromies are identities.
Substitution in the preceding formula gives the matrix.
For constant coefficients, $T_*=1$, hence $\can a=0$.
Thus $\can(\xi)$ depends only on $\rho(\xi)\in Q$.
Choose a compact loop $\gamma\subset C_!^\circ$ whose
supported class represents $1\in Q$. Its lift to a nearby
fiber gives $\xi$ with $\rho(\xi)=1$, so
$\can\xi=\can\ell(1)$.
On a tubular neighborhood of $\gamma$, the normal
$!$-extension is a product. The normal $\can$ map,
followed by forgetting support, is represented by a normal
circle. Hence $\varphi_I\can\ell(1)$ is the absolute torus
cycle $S^1_{\mathrm{normal}}\times\gamma$ and has zero
relative boundary. The first basis vector is $[P]$;
the second column therefore has boundary $3\partial[P]$.
\end{proof}

\section{Pentagram relations}
\label{sec:pentagram-relations}

\subsection{The Hodge comparison for one projection}
\label{subsec:one-projection}

Fix $p=p_3$. The two Hodge comparisons give a solid path in
\eqref{eq:pentagram-diagram}.

\begin{definition}
Put
\[
\mathbb F=p_*\jjj{\mathbb V}[2],
\qquad \boldsymbol\Phi=\Phi_{1-t}(\mathbb F).
\]
Let $\mathbb L_*,\mathbb L_!$ be the normal nearby-cycle
variations of Subsection~\ref{subsec:projection-vanishing-cycles},
with the same determinations, and put
\[
\mathbb E_*=H^1(C_*,j_{**!}\mathbb L_*)(-1),
\qquad
\mathbb E_!=H^1(C_!,j_{*!!}\mathbb L_!).
\]
\end{definition}

\begin{prop}
\label{prop:one-projection-hodge-comparison}
The vanishing cycles fit into an exact sequence of
mixed Hodge--Tate structures
\begin{equation}
0\longrightarrow\mathbb E_*
 \xrightarrow{\,\jmath\,}\boldsymbol\Phi
 \xrightarrow{\,\pi\,}\mathbb E_!
 \longrightarrow0.
\label{eq:component-hodge-extension}
\end{equation}
The structure $H^2(S,\jjj{\mathbb V})$ is Hodge--Tate,
and outer shrinking
\[
\varphi_I:\Phi\xrightarrow{\sim}H^2(S,\jjj{V})
\]
is strictly compatible with weights. If
$(\mathbb E_*)_B$ is identified with $E$ by the rational
Tate generator, then
\begin{equation}
U_*:=\sum_{n\ge0}\frac{(\log T_*)^n}{(n+1)!}
       =\frac{T_*-1}{\log T_*},
\label{eq:normal-variation-unit}
\end{equation}
\begin{equation}
\jmath_B=\iota U_*.
\label{eq:normalized-component-inclusion}
\end{equation}
The quotient in \eqref{eq:normal-variation-unit} means
the finite power series displayed there.
\end{prop}

\begin{proof}
Logarithmic localization gives
\eqref{eq:component-hodge-extension}, with the twist from
normal variation. The component structures are Hodge--Tate by
\cite{MarkarianHodge}, using its dual $*!!$ construction
on $C_!$. The relation between classical and logarithmic
variation in \cite[Subsection~2.1]{MarkarianHodge} gives
\eqref{eq:normalized-component-inclusion}.

Smooth-fiber cohomology is concentrated in degree one
and is exact on the coefficient variations. For
$\mathbb Q(n)$ its constant Tate subquotients are
$\mathbb Q(n)$ and $\mathbb Q(n-1)$, from the relative
cohomology of
$(\mathbb P^1\setminus\{0,1\},\{\lambda,\infty\})$.
Exactness and induction on a unipotent filtration give
constant Tate weight-graded pieces for the smooth part of
$\mathbb F$. The logarithm of base monodromy lowers
weights by at least two, so the
relative monodromy filtration is the weight filtration.
The shrinking result of \cite[Section~1]{MarkarianDefect}
now gives the remaining assertions.
\end{proof}

\begin{definition}[Hodge defect]
For a Hodge--Tate structure $K$, let
\[
s_K^F:\operatorname{gr}^W K_{B,\mathbb C}
 \xrightarrow{\sim}K_{B,\mathbb C}
\]
be its Hodge splitting: on weight $2a$ it is the inverse
of $F^a\cap W_{2a}\to\operatorname{gr}_{2a}^W$.
For a strictly weight-compatible Betti isomorphism
$u:K_B\to L_B$, define
\[
u^F=s_L^F(\operatorname{gr}^W u)(s_K^F)^{-1},
\qquad \mathcal G_u=(u^F)^{-1}u.
\]
We complexify Betti maps without changing notation.
The map $u^F$ preserves both filtrations, and
$\mathcal G_u$ induces the identity on associated graded.
For outer shrinking, write
$\mathcal G_{\mathrm{out}}=\mathcal G_{\varphi_I}
=\mathcal G_I(\mathbb F|_{\mathbb C^*})$.
\end{definition}

\begin{definition}[Associator conventions]
Let $A=G_1^{-1}G_0$ be the associator for
\[
dG=\left(e_0\frac{dt}{t}+e_1\frac{dt}{t-1}\right)G,
\]
where $G_0(t)t^{-e_0}\to1$ at $0$ and
$G_1(t)(1-t)^{-e_1}\to1$ at $1$, along $(0,1)$.
If $A_{\mathrm H}$ denotes the associator in the
$dt/(1-t)$ convention of \cite{MarkarianHodge}, then
$A(e_0,e_1)=A_{\mathrm H}(e_0,-e_1)$, as in
\cite[Subsection~2.3]{MarkarianDefect}.
For a series $g$ with constant term $1$, write
\[
g=1+g_0e_0+g_1e_1,\qquad g_Y=1+g_1e_1.
\]
Thus $g_Y$ retains the constant term and the words
ending in $e_1$.
\end{definition}

\begin{prop}
\label{prop:outer-hodge-defect}
In the right transport action $\mathscr R_{\mathbb F}$
of \cite[Section~2]{MarkarianDefect}, the outer defect is
\begin{equation}
\mathcal G_{\mathrm{out,dR}}
 =\mathscr R_{\mathbb F}(g_I),
\qquad
g_I=U(e_1)A(-e_0,-e_1)_Y,
\quad
U(e_1)=\sum_{n\ge0}\frac{(2\pi i e_1)^n}{(n+1)!}.
\label{eq:outer-associator-defect}
\end{equation}
All series act by finite sums.
\end{prop}

\begin{proof}
Apply the defect formula of \cite[Section~2]{MarkarianDefect}
using Proposition~\ref{prop:one-projection-hodge-comparison}
and the right-action convention.
\end{proof}

\subsection{The pentagram diagram}
\label{subsec:all-projections}

Put $r=\operatorname{rank}V$.
\begin{definition}[Corner and component spaces]
Number the sides $C_i$ of the real pentagon $P$ cyclically,
with indices in $\mathbb Z/5\mathbb Z$, as follows:
\[
(C_0,C_1,C_2,C_3,C_4)
 =(D_{15},D_{23},D_{45},D_{12},D_{34}).
\]
These are the $*$-components of the mixed fibers for
the projections
\[
(p_3,p_5,p_2,p_4,p_1),
\]
respectively.
Put $q_i=C_{i-1}\cap C_i$; in particular,
$q_0=D_{15}\cap D_{34}$.

At $q_0=(1,0)$ choose the unit tangent vectors in the
coordinates $(1-x,y)$, pointing into $P$, and transport
them to the other corners by cyclic relabelling.
Let $V^{\mathrm{tan}}_{q_i}$ be the resulting tangential
coefficient fiber, and put
\[
W_i=V^{\mathrm{tan}}_{q_i}\otimes\mathfrak o_i,
\]
where $\mathfrak o_i$ is the normal orientation line of
$P$ at that corner, understood from the interior.
We trivialize it by the orientation of
Subsection~\ref{subsec:two-dimensional-shrinking}.
Exchanging the two degree-one support maps acts on
$\mathfrak o_i$ by $-1$. All endpoint and reflection maps
include this sign.
Let $L_i$ be the normal nearby-cycle local system on
$C_i^\circ$ for the corresponding projection, and put
$E_i=H^1(C_i,j_{**!}L_i)$. The component inclusion and
outer shrinking give an embedding
\[
b_i:E_i\hookrightarrow H^2(S,\jjj{V}).
\]
The two endpoint shrinkings on $C_i$, with classical
Betti variation used to identify the endpoint fibers,
give isomorphisms
\[
a_i:W_i\xrightarrow{\sim}E_i,
\qquad
a_i':W_{i+1}\xrightarrow{\sim}E_i.
\]
Thus $\dim W_i=\dim E_i=r$ and $\dim H^2(S,\jjj{V})=2r$.

\end{definition}

\begin{definition}[Pentagram diagram]
The dashed arrow along $C_i$ is its tangential holonomy
\[
h_i:W_{i+1}\xrightarrow{\sim}W_i.
\]
These maps form the pentagram diagram \eqref{eq:pentagram-diagram}.

\newlength{\pentagramSideLength}
\setlength{\pentagramSideLength}{4.5cm}

\begin{equation}
\label{eq:pentagram-diagram}
\begin{tikzpicture}
\node (P0) at (90:0) {$H^2(S,\jjj{V})$};
\node (P1) at (90:.8506508084\pentagramSideLength) {$W_0$};
\node (P2) at (90+72:.8506508084\pentagramSideLength) {$W_1$};
\node (P3) at (90+2*72:.8506508084\pentagramSideLength) {$W_2$};
\node (P4) at (90+3*72:.8506508084\pentagramSideLength) {$W_3$};
\node (P5) at (90+4*72:.8506508084\pentagramSideLength) {$W_4$};
\node (P6) at (90+36:.525731112\pentagramSideLength) {$E_0$};
\node (P7) at (90+108:.525731112\pentagramSideLength) {$E_1$};
\node (P8) at (90+180:.525731112\pentagramSideLength) {$E_2$};
\node (P9) at (90+252:.525731112\pentagramSideLength) {$E_3$};
\node (P10) at (90+324:.525731112\pentagramSideLength) {$E_4$};
\path[commutative diagrams/.cd, every arrow, every label]
(P2) edge [dashed] (P1)
(P3) edge [dashed] (P2)
(P4) edge [dashed] (P3)
(P5) edge [dashed] (P4)
(P1) edge [dashed] (P5)
(P6) edge (P0)
(P7) edge (P0)
(P8) edge (P0)
(P9) edge (P0)
(P10) edge (P0)
(P1) edge (P10)
(P1) edge (P6)
(P2) edge (P6)
(P2) edge (P7)
(P3) edge (P7)
(P3) edge (P8)
(P4) edge (P8)
(P4) edge (P9)
(P5) edge (P9)
(P5) edge (P10);
\end{tikzpicture}
\end{equation}
\end{definition}

\begin{theorem}\label{thm:pentagram-geometric}
The five solid parallelograms in \eqref{eq:pentagram-diagram} commute:
\begin{equation}
b_i a_i=b_{i-1}a_{i-1}'\colon W_i\longrightarrow H^2(S,\jjj{V}).
\label{eq:solid-pentagram-relations}
\end{equation}
Their common maps are embeddings with image $L_P(V)$.
\end{theorem}

\begin{proof}
Let $\epsilon_i:W_i\xrightarrow{\sim}S_P(V)$ spread a
tangential value over $P$. Applying
Lemma~\ref{lem:main-pentagon-compatibility} along either
side through $q_i$ gives
\[
b_i a_i=\varphi_P\epsilon_i
       =b_{i-1}a_{i-1}'.
\]
Now use the injectivity of $\varphi_P$ from
Proposition~\ref{prop:two-dimensional-shrinking}.
\end{proof}

\begin{definition}[Dihedral relabelling]
We call \eqref{eq:solid-pentagram-relations} the
\emph{pentagram relations}. The rotation $\tau=(13524)$
sends $C_i$ to $C_{i+1}$. The reflection $\sigma=(13)(45)$
fixes $q_0$ and exchanges $C_0=D_{15}$ with $C_{-1}=D_{34}$,
and hence $p_3$ with $p_1$.
Put
\[
\mathcal S=b_0a_0:W_0\longrightarrow H^2(S,\jjj{V}).
\]
Relabelling acts on the coefficient object as well.
For $V^\sigma=\sigma^*V$, let
\[
\begin{aligned}
R_\sigma^W&:W_0(V)\xrightarrow{\sim}W_0(V^\sigma),\\
R_\sigma^2&:H^2(S,\jjj{V})\xrightarrow{\sim}H^2(S,\jjj{V^\sigma})
\end{aligned}
\]
be the induced identifications, including the coefficient
paths, orientations, and variation conventions. Define
\[
(\sigma\cdot\mathcal S)_V
 =(R_\sigma^2)^{-1}\mathcal S_{V^\sigma}R_\sigma^W.
\]
No dihedral equivariant structure on $V$ is required.
\end{definition}

\begin{cor}
The relation at $q_0$ is
\begin{equation}
\mathcal S=\sigma\cdot\mathcal S.
\label{eq:reflection-pentagram-relation}
\end{equation}
The other four relations are its rotations.
\end{cor}
\begin{proof}
The reflection exchanges the two solid paths through $q_0$;
their rotations exhaust the ten composable pairs. Apply
Theorem~\ref{thm:pentagram-geometric}.
\end{proof}

\begin{rem}
The corner source is realized by iterated vanishing cycles,
retaining both normal Tate twists. Passing from the classical
Betti identifications to this realization uses the two variation
conversions \eqref{eq:normal-variation-unit}. They commute because
the normal monodromies commute. Thus the equations hold in
the Hodge comparison frames as well. The resulting
$2r\times r$ matrix transforms under a change of frames as
\[
\mathcal S\longmapsto C_{\mathrm{coh}}^{-1}\mathcal S C_W.
\]
The reflection identifications transform by the same rule.
For the completed regular representation, read the equations
on its finite unipotent quotients.
\end{rem}

\subsection{The associator formula}
\label{subsec:associator-formula}

\begin{definition}[Comparison coordinates]
We now use the corner and component identifications
supplied by logarithmic variation, as in the preceding remark.
Fix a corner on the $*$-component.
Write
\[
a:W_q\xrightarrow{\sim}(\mathbb E_*)_B,
\qquad
\jmath:(\mathbb E_*)_B\hookrightarrow\Phi,
\qquad
\varphi:\Phi\xrightarrow{\sim}H^2(S,\jjj{V}).
\]
Here $a,\varphi$ are component and outer shrinking,
$W_q$ retains the normal Tate twist, and $\jmath$ is
\eqref{eq:normalized-component-inclusion}.
The solid composite is $\mathcal S=\varphi\jmath a$.
Using their Hodge-compatible counterparts, put
\[
J=\jmath a^F:W_{q,\mathbb C}\hookrightarrow\Phi_{\mathbb C}.
\]
Choose Hodge frames adapted to this injection, so that
$J=\binom{1_r}{0}$.
\end{definition}

\begin{definition}[Transport actions]
Write, for $\nu\in\{\mathrm{comp},\mathrm{out}\}$,
\[
c_\nu:B_\nu\longrightarrow F_\nu,
\qquad v_\nu:F_\nu\longrightarrow B_\nu,
\qquad X_\nu\in\operatorname{End}(B_\nu),
\qquad Y_\nu=v_\nu c_\nu.
\]
Here $F_{\mathrm{comp}}=W_{q,\mathbb C}$,
$F_{\mathrm{out}}=\Phi_{\mathbb C}$, and $B_\nu$ is
the corresponding nearby-cycle space. These are the
logarithmic gluing coordinates of
\cite[Subsection~2.3]{MarkarianDefect}:
$X_\nu=N_{0,\nu}$ and $Y_\nu=-N_{1,\nu}$.
Let $\rev$ be the word-reversal anti-automorphism.
On $\mathbb C\cdot1\oplus
\mathbb C\langle\!\langle e_0,e_1\rangle\!\rangle e_1$,
define the right transport action, for every word $w$, by
\begin{equation}
R_\nu(we_1)
 =c_\nu\rev(w)(X_\nu,Y_\nu)v_\nu,
\qquad R_\nu(1)=1.
\label{eq:solid-transport-action}
\end{equation}
Thus $R_\nu(uv)=R_\nu(v)R_\nu(u)$.
\end{definition}

\begin{prop}
\label{prop:solid-composite-formula}
In the target frame transported by $\varphi^F$, the
matrix of the solid composite is
\begin{equation}
\widehat{\mathcal S}
 =(\varphi^F)^{-1}\mathcal S
 =G_{\rm out}\,J\,G_{\rm comp},
\qquad
G_{\rm out}=(\varphi^F)^{-1}\varphi,
\quad G_{\rm comp}=(a^F)^{-1}a.
\label{eq:solid-composite-defects}
\end{equation}
For $g_I$ of \eqref{eq:outer-associator-defect},
\begin{equation}
\widehat{\mathcal S}
 =R_{\mathrm{out}}(g_I)\,J\,R_{\mathrm{comp}}(g_I).
\label{eq:solid-composite-associator}
\end{equation}

\end{prop}
\begin{proof}
Since $\jmath$ is a Hodge morphism, substitution gives
\eqref{eq:solid-composite-defects}. Apply
\eqref{eq:outer-associator-defect} to its two factors.
Only the inclusion is used, so no lift of $\mathbb E_!$
is required.
\end{proof}

\begin{definition}[Residue coordinates]
With $\nabla=d-\Omega$, put
$r_{ij}=r_{ji}=\operatorname{Res}_{D_{ij}}\Omega$;
thus $r_{ij}$ is the negative of the connection residue.
We use the spherical presentation
\[
r_{ii}=0,\qquad \sum_jr_{ij}=0,\qquad
[r_{ij},r_{kl}]=0\quad
\text{if }\{i,j\}\cap\{k,l\}=\varnothing.
\]
All series below take values in the completed enveloping
algebra of this Lie algebra. Set
\begin{equation}
a=r_{12},\quad b=r_{23},\quad c=r_{34},\quad
d=r_{45},\quad e=r_{15}.
\label{eq:five-residue-generators}
\end{equation}
Eliminating the other five generators gives
\[
\begin{gathered}
r_{13}=d-a-b,\quad r_{14}=b-d-e,\quad r_{24}=e-b-c,\\
r_{25}=c-a-e,\quad r_{35}=a-c-d.
\end{gathered}
\]
The relations among $a,b,c,d,e$ are
\begin{equation}
\begin{gathered}
{}[a,c]=[a,d]=[b,d]=[b,e]=[c,e]=0,\\
[a,b]+[b,c]+[c,d]+[d,e]+[e,a]=0.
\end{gathered}
\label{eq:five-residue-relations}
\end{equation}
The reflection $\sigma=(13)(45)$ acts by
\begin{equation}
\sigma(a,b,c,d,e)=(b,a,e,d,c).
\label{eq:explicit-residue-reflection}
\end{equation}

\end{definition}

\begin{lemma}
In the universal logarithmic realization, the gluing
maps for $p_3$, in the order of the component and quotient
coordinates, are
\begin{equation}
\begin{gathered}
X=\begin{pmatrix}a&0\\1&d\end{pmatrix},\qquad
Y=\begin{pmatrix}e&c(d-a-b)\\0&e-b-c\end{pmatrix},\\
c_{\mathrm{out}}=\begin{pmatrix}e&0\\0&1\end{pmatrix},
\qquad
v_{\mathrm{out}}=\begin{pmatrix}1&c(d-a-b)\\0&e-b-c\end{pmatrix}.
\end{gathered}
\label{eq:explicit-outer-gluing}
\end{equation}
For the component at the corner $D_{15}\cap D_{34}$,
\begin{equation}
X_{\mathrm{comp}}=b,\qquad Y_{\mathrm{comp}}=c,
\qquad c_{\mathrm{comp}}=c,\qquad v_{\mathrm{comp}}=1,
\qquad J=\binom10.
\label{eq:explicit-component-gluing}
\end{equation}
Here $X=X_{\mathrm{out}}$ and $Y=Y_{\mathrm{out}}$.
\end{lemma}

\begin{proof}
The relative logarithmic calculation is given in
Appendix~\ref{app:logarithmic-direct-image}.
\end{proof}

\begin{prop}[The explicit column]
\label{prop:explicit-pentagram-column}
For a series $A$ with constant term $1$, put
\[
A=1+A_0e_0+A_1e_1,
\qquad \mathscr F_A(u,v)=\rev(A_1)(u,v).
\]
Thus $\mathscr F_A$ deletes the final $e_1$ and reverses
the remaining word. The raw semi-holonomy composite is
\begin{equation}
S_A=
\begin{pmatrix}
1+e\,\mathscr F_A(X,Y)_{11}\\
\mathscr F_A(X,Y)_{21}
\end{pmatrix}
\bigl(1+c\mathscr F_A(b,c)\bigr).
\label{eq:explicit-pentagram-column}
\end{equation}
\end{prop}
\begin{proof}
Substitute \eqref{eq:explicit-outer-gluing} and
\eqref{eq:explicit-component-gluing} in the transport
formula \eqref{eq:solid-transport-action}.
Every nonzero $(i,j)$ entry of a word of length $m$ in
$X,Y$ has residue degree $m+j-i$. Thus the substitution
is continuous for the degree completions, despite
$X_{21}=1$.
\end{proof}

\begin{rem}
$\mathcal S$ denotes the intrinsic solid map,
$\widehat{\mathcal S}$ its Hodge comparison matrix, and
$S_A$ the raw associator column; the variation factors and
the sign change relating the last two are recorded in the
proof of Theorem~\ref{thm:pentagram-coefficients}.
The two entries of $S_A$ are denoted by $S_A^{(1)}$ and $S_A^{(2)}$.
They are $r\times r$ blocks in comparison coordinates. In the two-disk Betti
presentation, the supported map is
$\varphi_{P,\Sigma}^{-1}\varphi_P(v)=(v,0)$; passing to
the de Rham frames produces the two blocks above.
\end{rem}

\subsection{The two families of relations}
\label{subsec:coefficient-relations}

\begin{definition}
\label{def:pentagram-coefficient-relations}
Let $A\in\mathbb C\langle\!\langle e_0,e_1\rangle\!\rangle$
be group-like for the coproduct with primitive generators
$e_0,e_1$. We call it \emph{normalized} if its constant
term is $1$ and its linear terms vanish. Define $S_A$
by \eqref{eq:explicit-pentagram-column}.
The \emph{upper pentagram relations} are the coefficient
equations of
\begin{equation}
S_A^{(1)}(a,b,c,d,e)=S_A^{(1)}(b,a,e,d,c),
\label{eq:upper-pentagram-relations}
\end{equation}
and the \emph{lower pentagram relations} are those of
\begin{equation}
S_A^{(2)}(a,b,c,d,e)=S_A^{(2)}(b,a,e,d,c).
\label{eq:lower-pentagram-relations}
\end{equation}
Together they are the \emph{pentagram coefficient relations}.
No pentagon equation is part of this definition.
\end{definition}

\begin{lemma}
\label{lem:universal-graded-change}
The universal logarithmic frames of
\eqref{eq:explicit-outer-gluing} identify the graded
de Rham vanishing-cycle space with
$V_{\mathrm{dR}}(-2)\oplus V_{\mathrm{dR}}$.
After transporting the coefficient action by $\sigma$,
the corner and cohomology frame transitions are
\[
C_\sigma^W=\varepsilon\,1,\qquad
C_\sigma^2=\varepsilon\,1,
\qquad \varepsilon=-1.
\]
The sign is the exchange of the two normal orientations.
\end{lemma}
\begin{proof}
Appendix~\ref{app:reflection-frames} determines the transition
from the grading and geometric commutativity for the period
associator, without using the coefficient theorem below.
\end{proof}

\begin{theorem}[The pentagram coefficient relations]
\label{thm:pentagram-coefficients}
The period associator satisfies both families:
\begin{equation}
S_A(a,b,c,d,e)=S_A(b,a,e,d,c).
\label{eq:explicit-mzv-reflection}
\end{equation}
The equality is in the completed algebra with relations
\eqref{eq:five-residue-relations}.
\end{theorem}

\begin{proof}
The geometric Theorem~\ref{thm:pentagram-geometric},
expressed in arbitrary comparison frames, gives
\begin{equation}
\widehat{\mathcal S}(V)
 =(C_\sigma^2)^{-1}\widehat{\mathcal S}(V^\sigma)C_\sigma^W,
\label{eq:framed-pentagram-relation}
\end{equation}
where $C_\sigma^2,C_\sigma^W$ represent
$R_\sigma^2,R_\sigma^W$.

By Lemma~\ref{lem:universal-graded-change} the source
and target signs cancel. Put $\widetilde A=A(-e_0,-e_1)$.
The Hodge column is $S_{\widetilde A}U(e)U(c)$, as follows
from the transport formula and
$(c_{\mathrm{out}}v_{\mathrm{out}})J=Je$.
The right factor is invertible and invariant under
$\sigma$, since $ec=ce$. Cancelling it gives the raw
reflection identity for $\widetilde A$.
Simultaneously negating the five residues gives the
identity for $A$. For the outer matrices this substitution
uses conjugation by $\operatorname{diag}(-1,1)$; it
commutes with $\sigma$.
\end{proof}

\begin{cor}
Both families are polynomial equations of degree at most two,
with rational coefficients, in the coefficients of $A$.
\end{cor}
\begin{proof}
Both factors in \eqref{eq:explicit-pentagram-column} are
linear in the coefficients of $A$ and have rational
residue coefficients.
\end{proof}

\begin{rem}
The superscripts distinguish the two blocks, not their
weights. Associator weight $n$ occurs in residue degree
$n$ in $S_A^{(1)}$ and in degree $n-2$ in $S_A^{(2)}$.
The coefficient extraction and all calculations through
weight five are collected in
Appendix~\ref{app:coefficient-calculations}, with the
upper and lower equations recorded separately.
The frame calculation above uses the period associator
to determine a transition independent of $A$; the formal
relations also make sense when the quadratic coefficient
of $A$ is zero.
\end{rem}

\section{Pentagon, convolution, and shuffle}
\label{sec:consequences-specializations}

\subsection{From pentagram to pentagon}
\label{subsec:pentagram-to-pentagon}

\begin{lemma}[Endpoint compatibility]
Semi-holonomy and holonomy along a side satisfy
\begin{equation}
a_i h_i=a_i'.
\label{eq:edge-holonomy-compatibility}
\end{equation}
\end{lemma}
\begin{proof}
See \cite[Subsection~3.1]{MarkarianHodge} for the standard
path and \cite[Section~6]{MarkarianConvolution} for
arbitrary pro-unipotent paths.
\end{proof}

\begin{cor}\label{cor:pentagram-pentagon}
The dashed holonomies in \eqref{eq:pentagram-diagram} satisfy
\begin{equation}
h_0h_1h_2h_3h_4=1_{W_0}.
\label{eq:holonomy-pentagon}
\end{equation}
In associator coordinates this is the pentagon equation.
More generally, the pentagram coefficient relations imply
the five-factor equation
\begin{equation}
A(b,c)A(d,e)A(a,b)A(c,d)A(e,a)=1
\label{eq:formal-five-factor-pentagon}
\end{equation}
in the completed residue algebra
\eqref{eq:five-residue-relations}, for normalized group-like series
satisfying the reversal relation
\begin{equation}
A(e_0,e_1)A(e_1,e_0)=1.
\label{eq:formal-reversal}
\end{equation}
\end{cor}

\begin{proof}
Put $J_i=b_i a_i=b_{i-1}a_{i-1}'$. The endpoint identity gives
\[
J_i h_i=b_i a_i h_i=b_i a_i'=J_{i+1}.
\]
Hence $J_0h_0h_1h_2h_3h_4=J_0$, and injectivity of
$J_0$ proves \eqref{eq:holonomy-pentagon}.

For a formal series, the endpoint identity of
Appendix~\ref{app:formal-endpoint-identity}, together with
the coefficient relations, gives
\[
S_A A(-b,-c)=K_\Phi\tau(S_A),\qquad
K_\Phi=\begin{pmatrix}1&e(a-d)\\0&1\end{pmatrix}.
\]
Since $\prod_{i=0}^4\tau^i(K_\Phi)=1$ by the same appendix,
iteration gives
\[
S_A\prod_{i=0}^4\tau^i\bigl(A(-b,-c)\bigr)=S_A,
\]
with factors ordered from left to right by increasing $i$.
The upper entry
of $S_A$ has constant term $1$ and is invertible.
Cancel it and negate the five residues to obtain
\eqref{eq:formal-five-factor-pentagon}.
\end{proof}

\begin{cor}\label{cor:lower-pentagram-pentagon}
Let $A\in\mathbb C\langle\!\langle e_0,e_1\rangle\!\rangle$
be normalized and group-like, and satisfy
\eqref{eq:formal-reversal}. If $\kappa=[e_0e_1]A\ne0$,
the lower pentagram relations alone imply the pentagon
equation. Under these hypotheses, the lower relations,
the full pentagram relations, and the pentagon equation
are equivalent.
\end{cor}

\begin{proof}
Put $L_A=S_A^{(2)}$. Its constant term in residue degree
is $\kappa$: in \eqref{eq:explicit-pentagram-column},
only the term $\kappa X$ of $\mathscr F_A(X,Y)$ contributes
to this constant term, through $X_{21}=1$.
Thus $L_A$ is invertible in the completed residue algebra.
Taking the lower row of
\eqref{eq:formal-column-endpoint} gives
\[
L_A A(-b,-c)=\tau\sigma(L_A).
\]
The lower reflection relation $\sigma(L_A)=L_A$ therefore
gives
\begin{equation}
A(-b,-c)=L_A^{-1}\tau(L_A).
\label{eq:lower-pentagon-factor}
\end{equation}
Since $\tau^5=1$, the five rotated factors telescope:
\[
\prod_{i=0}^4\tau^i\bigl(A(-b,-c)\bigr)
 =\prod_{i=0}^4\Bigl(\bigl(\tau^i(L_A)\bigr)^{-1}
                  \tau^{i+1}(L_A)\Bigr)=1,
\]
with factors ordered by increasing $i$ from left to right.
Negating the residues gives
\eqref{eq:formal-five-factor-pentagon}.
The equivalence with the full pentagram relations then
follows from Theorem~\ref{thm:pentagon-pentagram} below.
\end{proof}

\begin{rem}
For the period associator, $\kappa=-\zeta(2)\ne0$.
Over a commutative coefficient algebra, the same implication
from the lower relations to pentagon holds when $\kappa$
is a unit, or after localization at $\kappa$.
\end{rem}

\subsection{From pentagon to pentagram}
\label{subsec:pentagon-to-pentagram}

For the converse, we replace the real pentagon by a relative
class with prescribed pro-unipotent boundary paths. Work over
$\mathbb C$, with the tangential base points, orientations,
and variation conventions of Subsection~\ref{subsec:all-projections}.

\begin{definition}[Boundary faces]
Let $\widetilde S$ be the real oriented blow-up of $S$
along its boundary. Write $\partial_*\widetilde S$ and
$\partial_!\widetilde S$ for the unions of the faces over
$D_*$ and $D_!$. The inclusion
$U\hookrightarrow\widetilde S$ is a homotopy equivalence;
write $\widetilde V$ for the extension of $V$ to
$\widetilde S$.
\end{definition}

\begin{prop}
\label{prop:film-relative-realization}
The mixed-extension realization and the complex orientation
of $S$ give natural isomorphisms
\begin{equation}
\begin{aligned}
H^2(S,\jjj{V})
&\simeq H^2(\widetilde S,\partial_!\widetilde S;
                     \widetilde V)\\
&\simeq H_2(\widetilde S,\partial_*\widetilde S;
                     \widetilde V).
\end{aligned}
\label{eq:film-relative-realization}
\end{equation}
With the normal orientations of Section~\ref{sec:betti-presentations},
the lifted real pentagon corresponds to its supported
cohomology class.
\end{prop}
\begin{proof}
The normal-circle model, retaining normal monodromy,
computes the $*$-extension; relative cochains at the
complementary faces compute $!$-extension.
Poincar\'e--Lefschetz duality gives the second isomorphism
and identifies the pentagon class with its support map.
\end{proof}

\begin{definition}[Universal path coefficients]
Fix the tangential corner $q=q_0$. Write $\mathscr R_q$
for the completed universal path local system: its fiber
at $z$ is the completed linear span of paths from $q$
to $z$. Its fiber at $q$ is the completed group algebra
of $\pi_1(U,q)$. For $v\in V_q$, evaluation of transport
gives a coefficient map
\[
\mathscr R_q\longrightarrow\widetilde V,
\qquad [\gamma]\longmapsto\gamma(v).
\]
Fix a finite cellular model for
$(\widetilde S,\partial_*\widetilde S)$. Complexes with
coefficients in $\mathscr R_q$ mean the termwise
augmentation completions of the cellular coefficient
complexes; their homology is taken after completion.
On a boundary face, first use its own group algebra and
augmentation filtration, then extend coefficients to
$\mathscr R_q$ by completed tensor product. Reference paths identify path torsors
with these group algebras. Normalized bar tensors are
completed for total augmentation degree.
\end{definition}

\begin{lemma}
\label{lem:universal-relative-film}
The boundary map is an isomorphism
\[
\partial:H_2(\widetilde S,\partial_*\widetilde S;
                         \mathscr R_q)
\xrightarrow{\sim}
H_1(\partial_*\widetilde S;\mathscr R_q).
\]
\end{lemma}

\begin{proof}
As a fiber-type arrangement complement, $\M{5}$ is a
rational $K(\pi,1)$. More precisely, the continuous dual
$\mathscr O_q=\mathscr R_q^\vee$ is the ind-unipotent local
system of path-torsor functions. Its de Rham complex is
$\Omega^\bullet B(U)$, whose positive-degree cohomology
vanishes by \cite[Corollary~3.40]{BrownModuli}.
Comparison on finite-dimensional subobjects identifies
this with cellular cohomology: filtered colimits commute
with the finite cellular calculation. Exact continuous
duality between discrete and linearly compact vector
spaces therefore gives
\[
H_1(\widetilde S;\mathscr R_q)
=H_2(\widetilde S;\mathscr R_q)=0.
\]
The exact sequence of the pair proves the assertion.
This vanishing holds before evaluation in a finite local
system; it need not survive that evaluation.
\end{proof}

\begin{definition}[Boundary paths]
Let $A$ be a group-like series with constant term $1$
and zero linear terms, satisfying
\eqref{eq:formal-reversal} and
\eqref{eq:formal-five-factor-pentagon}.
Its five tangential substitutions specify de Rham paths
$\gamma_i^A$ on the faces over $C_i$, directed as the dashed
arrows. A fixed Betti--de Rham comparison makes them
completed Betti paths. In the directed endpoint frames,
their transport series is $A(-e_0,-e_1)$, whose substitutions
are the factors in the proof of
Corollary~\ref{cor:pentagram-pentagon}.
\end{definition}

\begin{rem}
The normal meridian of a face is central. Changing the
splitting into normal and tangential generators adds central
terms to the arguments of $A$, leaving it unchanged since
$\log A$ has zero linear part. Fix the inward normal directions
at the corners, with no independent normal winding.
\end{rem}

\begin{lemma}[Completed path operations]
\label{lem:completed-path-operations}
The Fox representatives of face paths, their comparison
with normalized bar chains, and the transport prisms
extend continuously to the completions above. The same
holds for the interval operations defining semi-holonomy.
After extending coefficients, their chain identities
commute with evaluation in a finite unipotent local system.
In particular, the cellular image of a bar representative
differs from the Fox representative of the same completed
face path by a boundary in that face.
\end{lemma}
\begin{proof}
Trivialize the normal circle bundle over each punctured
component and use the cellular model
$S^1\times(S^1\vee S^1)$. Its resolution is the product
of the circle resolution and the free-group Fox resolution.
On either factor, a Fox derivative satisfies
$D(\mathfrak a^n)\subset\mathfrak a^{n-1}$ for its augmentation ideal $\mathfrak a$
and $n\geq1$.
The comparison with bar chains in degrees one and two,
and the homotopy comparing path representatives, use
finitely many Fox operators and multiplication. They
therefore have bounded loss of augmentation degree and
extend continuously. Choose these maps before extending
coefficients; their chain and homotopy identities are
preserved by this extension and by evaluation.

The interval operations are continuous by the same estimate
and the Fox formula of \cite[Section~6]{MarkarianConvolution}.
The bar formula for a transport prism uses multiplication
and two bar factors. Fiber monodromy fixes the punctures
and acts trivially on their first homology; its difference
from the identity raises fiber augmentation degree.
Thus transport also extends to completed base paths.
Repeated occurrences of a path are linearized by the
coproduct, and inverses by the antipode, both continuous.
The chain identities for ordinary paths consequently
extend from the dense group algebras to their completions.
After passage to the finite cellular model, evaluation in
a finite unipotent local system annihilates a sufficiently
high augmentation power in the coefficients and preserves
these identities.
\end{proof}

\begin{definition}[Relative film]
Use Lemma~\ref{lem:completed-path-operations} to represent
the paths by completed chains in their boundary faces.

Transport the unit at $q$ around the oriented boundary,
obtaining a chain $z_A$ in $\partial_*\widetilde S$.
Its boundary is total transport minus the identity,
applied to that unit. Thus
\[
\partial z_A=0
\]
is the pentagon equation, and
Lemma~\ref{lem:universal-relative-film} defines a unique
relative class
\[
[\Gamma_A]=\partial^{-1}[z_A].
\]
Call this class the \emph{relative film} of $A$. Evaluation
in $V$, via \eqref{eq:film-relative-realization}, gives
\begin{equation}
s_{A,q}:W_q\longrightarrow H^2(S,\jjj{V}).
\label{eq:film-supported-map}
\end{equation}
Here $W_q=V_q\otimes\mathfrak o_q$ includes the normal
orientation line. The film is a completed chain class;
no embedded disk representing its boundary is required.
\end{definition}

Replace component and outer shrinking by semi-holonomy
along the paths specified by $A$. Denote the resulting
solid maps by $a_i^A,(a_i')^A,b_i^A$. The component
inclusion in the definition of $b_i^A$ is unchanged.

\begin{lemma}
\label{lem:film-successive-shrinking}
The two solid composites at $q_i$ are the evaluation
of the same relative film:
\[
b_i^Aa_i^A=s_{A,q_i}
          =b_{i-1}^A(a_{i-1}')^A.
\]
\end{lemma}

\begin{proof}
For $p_3$, the $*$-sections are $D_{23},D_{34}$,
the fiber over $0$ is $D_{12}\cup D_{45}$, and the
$*$-component over $1$ is $C_0=D_{15}$.
Specialize the component semi-holonomy path on $D_{15}$
to a nearby smooth fiber and sweep it along the outer
semi-holonomy path. Away from the singular endpoints,
this gives a rectangle with lateral edges on the
$*$-sections.

For the completed sweep, let $P$ be the initial fiber path,
$P'$ its transport, and $\beta_0,\beta_1$ the endpoint paths.
Writing $g\cdot h$ for concatenation in traversal order,
we have $P\cdot\beta_1=\beta_0\cdot P'$. In the normalized
path bar complex put
\[
K(P,\gamma)=[P\mid\beta_1]-[\beta_0\mid P'].
\]
The bar differential
$\partial[g\mid h]=g[h]-[g\cdot h]+[g]$ gives
\begin{equation}
\partial K(P,\gamma)
 =[P]+P[\beta_1]-[\beta_0]-\beta_0[P'].
\label{eq:completed-sweep-boundary}
\end{equation}
The four terms are the initial, endpoint, and final chains,
with transported coefficients. The formula holds for
completed paths by Lemma~\ref{lem:completed-path-operations}.

For ordinary paths, the prism and transport square have
the same boundary representatives. Their difference is an
absolute universal $2$-cycle, zero in homology by the
proof of Lemma~\ref{lem:universal-relative-film}.
The composition law for support maps
identifies the square with successive interval support
(Lemma~\ref{lem:main-pentagon-compatibility}).
Lemma~\ref{lem:completed-path-operations} extends this
comparison to completed paths, before imposing pentagon.
Thus
\eqref{eq:completed-sweep-boundary} also computes the
boundary of two-step semi-holonomy.

The prescribed terminal fiber path concatenates the paths
on $D_{12}$ and $D_{45}$, using the positive node bridge:
in $uv=t$, take $u,v>0$ over $t>0$. Together with the fixed
inward tangential directions, this specifies the normal
winding. To identify the transported path with this one,
apply the intrinsic path-torsor identity of
\cite[Proposition~19]{MarkarianConvolution}: for fiber path
$P$ and base path $\gamma$,
\[
i_*P\cdot s_{2*}(\gamma)
 =s_{1*}(\gamma)\cdot i_*\operatorname{Hol}_{\gamma}(P).
\]
The fiber pro-unipotent fundamental group injects into
that of $U$, so equality with the prescribed broken path
is precisely the five-path pentagon equation.
For the boundary substitutions, set $y=(1-t)x$ in the
connection of Appendix~\ref{app:logarithmic-direct-image}
and write its tangential part as
$R_0\,d\log z+R_1\,d\log(z-1)$. The normal and tangential
residues are
\[
\begin{array}{c|c|c|c}
\text{face}&z&n&(R_0,R_1)\\ \hline
D_{15}&x&e&(b,c)\\
D_{23}&t&b&(d,e)\\
D_{34}&t&c&(a-c,e)\\
D_{45}&x&d&(b,a-d)\\
D_{12}&w&a&(c,d-a)
\end{array}
\]
The last row uses $x=1+t\xi$, $w=\xi/(\xi-1)$.
In each row $n$ commutes with the tangential residues.
Removing its shifts gives the five prescribed substitutions,
with reversal for
backward traversal.

In the fixed collars write $P'=Q_1\cdot Q_2$ in traversal
order, and replace the prism's final edge by this broken edge:
\[
K'=K-\beta_0[Q_1\mid Q_2].
\]
Then
\[
\partial K'
 =[P]+P[\beta_1]-[\beta_0]
   -\beta_0[Q_1]-(\beta_0\cdot Q_1)[Q_2].
\]
This boundary is supported on the five components
\[
D_{15},\quad D_{23},\quad D_{45},\quad D_{12},\quad D_{34},
\]
with the orientations and normal directions of $z_A$.
Changing representatives within a face changes no relative class.

Specialization in the initial collar supplies the component
inclusion in $b_0^A$. The support comparison and
\cite[Section~6]{MarkarianConvolution} therefore identify
this chain, with relative boundary $[z_A]$, with $b_0^Aa_0^A$.
For $p_1$ the construction starts on $D_{34}$ and has the
same five oriented boundary paths. By
Lemma~\ref{lem:universal-relative-film}, both universal
relative classes equal $[\Gamma_A]$. Evaluation in $V$
and \eqref{eq:film-relative-realization} give the result.
Relabelling treats every corner; $\mathfrak o_i$ includes
the sign from exchanging the interval directions.
\end{proof}

\begin{theorem}
\label{thm:pentagon-pentagram}
For normalized group-like series satisfying
\eqref{eq:formal-reversal}, the pentagon equation and
the pentagram coefficient relations are equivalent.
\end{theorem}

\begin{proof}
For the converse to Corollary~\ref{cor:pentagram-pentagon},
Lemma~\ref{lem:film-successive-shrinking} gives reflection
in the full space $H^2(S,\jjj{V})$. Express the composite
in the fixed universal logarithmic frames of
Subsections~\ref{subsec:associator-formula}--\ref{subsec:coefficient-relations}.
Naturality in finite logarithmic gluing data writes raw
semi-holonomy as $1+c\,h(X,Y)v$, by the argument of
\cite[Section~2]{MarkarianDefect}. Specialization to the
$*$-extension, $v=1$, $c=e_1$, determines $h$, since
$h\mapsto e_1h$ is injective on universal corner series.
There the Fox formula of \cite[Section~6]{MarkarianConvolution}
gives the rightmost-letter expression for transport.
For $T_A=A(-e_0,-e_1)$, group-likeness gives
$\rev(T_A^{-1})=A$; hence either transport action is
\[
R_\nu\bigl((\rev(T_A^{-1}))_Y\bigr)=R_\nu(A_Y)
 =1+c_\nu\mathscr F_A(X_\nu,Y_\nu)v_\nu,
\]
where $A=1+A_0e_0+A_1e_1$.
The composition $R_{\mathrm{out}}(A_Y)J
R_{\mathrm{comp}}(A_Y)$ is exactly
\eqref{eq:explicit-pentagram-column}.
Frame transitions and normal variation conversions depend
only on the logarithmic coefficient object. The commuting
variation factors cancel as in
Theorem~\ref{thm:pentagram-coefficients}, giving
\[
S_A(a,b,c,d,e)=S_A(b,a,e,d,c).
\]
Finite regular quotients detect this identity in the
completed algebra. Both blocks occur because reflection
holds in the full relative cohomology group.
\end{proof}

\begin{rem}
The film is relative, so $\pi_2(U)=0$ causes no obstruction;
rational asphericity supplies the universal acyclicity
comparing its fillings. Standard boundary paths recover
the supported real pentagon. For general $A$, its image
is not identified with the original $L_P(V)$.
\end{rem}

\subsection{The convolution specialization}
\label{subsec:two-projections}

\begin{definition}\label{def:convolution-coefficient-object}
Use the target coordinates
\[
X=\frac1{1-y},\qquad Y=\frac{x-y}{x(1-y)}
\]
for $p_4,p_2$, respectively; then $p_3$ has coordinate
$t=Y/X$. For iterated Kummer extensions $\mathbb V_1,
\mathbb V_2$ on $\M{4}$, put
\[
\mathbb F_i=j_{**!}\mathbb V_i[1],\qquad
\mathbb V=p_4^*\operatorname{inv}^*\mathbb V_1
 \otimes p_2^*\mathbb V_2,\qquad
\operatorname{inv}(z)=z^{-1}.
\]
Let $\mathscr K_{\mathrm{conv}}$ be the extension of
$\mathbb V[2]$ with the boundary conditions of
Subsection~\ref{subsec:mixed-extension}, replacing the
$!$-extension along $D_{24}$ by the smooth extension.
\end{definition}

On $D_{24}^\circ$ one has $X=Y=1/(1-y)$; both
coefficient maps extend across this divisor with values
in $\M{4}$. Thus $\mathbb V$ extends smoothly across
$D_{24}^\circ$. Retaining the prescribed conditions on
the other boundary components, adjunction gives a natural morphism
\[
\jjj{\mathbb V}[2]\longrightarrow\mathscr K_{\mathrm{conv}}.
\]
It is an isomorphism near $P$, since $P\cap D_{24}=\varnothing$.
On Betti realizations, the induced map on cohomology
therefore carries $\varphi_P$ to the supported-class map
on $P$ with convolution coefficients. Naturality makes
this identification compatible with proper direct image
and the nearby- and vanishing-cycle maps. This is the
comparison with the original mixed extension used below.

\begin{prop}\label{prop:convolution-supported-square}
The two projections give
\[
Rp_{3*}\mathscr K_{\mathrm{conv}}\simeq\mathbb F_1*\mathbb F_2,
\qquad
Rp_{4*}\mathscr K_{\mathrm{conv}}
 \simeq\operatorname{inv}^*\mathbb F_1
       \otimes H^0(\mathbb P^1,\mathbb F_2).
\]
Their comparison is the convolution shrinking square of
\cite[Subsection~3.3]{MarkarianDefect}. Its two Betti
composites are the supported-class map on $P$.
\end{prop}

\begin{proof}
Contracting $D_{15}$ gives the blow-up $S'$ of
$\mathbb P^1_X\times\mathbb P^1_Y$ at $(0,0)$ and
$(\infty,\infty)$. For the contraction $\gamma:S\to S'$
and the blow-up map $b$, one has
\[
R\gamma_*\mathscr K_{\mathrm{conv}}
 \simeq b^*(\operatorname{inv}^*\mathbb F_1
                    \boxtimes\mathbb F_2).
\]
The two objects agree away from $(1,1)$ and are both
$*$-extensions across that point: the external product
is $*$-extended at $X=1$ and $Y=1$, and its costalk at
their intersection vanishes. The pushforward assertions
now follow from \cite[Subsection~3.1]{MarkarianDefect}.

Put $u=1/X$ and $v=Y$, so that $t=uv$. The inverse
coordinate change
\[
y=1-u,\qquad x=\frac{1-u}{1-uv}
\]
identifies $(0,1)^2$ with $P^\circ$, preserving the
orientation of Subsection~\ref{subsec:two-dimensional-shrinking}.
The $p_4$ route shrinks first in $v$, then in $u$;
the $p_3$ route pushes along $uv=t$, then shrinks in $t$.
Both forget support on $P$, identifying our Betti maps
with those of the cited square. The beta factor and
the Hodge defect identity are consequently those of
\cite[Subsections~3.2--3.3]{MarkarianDefect}.
\end{proof}

\begin{cor}\label{cor:pentagram-double-shuffle}
A normalized group-like series satisfying reversal and
the pentagram coefficient relations satisfies Racinet's
regularized double shuffle relations.
\end{cor}

\begin{proof}
For the period associator, the proposition identifies
the Hodge comparison with the one giving double shuffle
in \cite[Subsection~3.3]{MarkarianDefect}.
For an arbitrary such series, apply
Corollary~\ref{cor:pentagram-pentagon} and
\cite[Theorem~1.1]{FurushoDoubleShuffle}.
\end{proof}

\subsection{Shuffle relations}
\label{subsec:shuffle}

In this subsection the two equations
\eqref{eq:upper-pentagram-relations} and
\eqref{eq:lower-pentagram-relations} are considered for arbitrary
formal series, without the group-likeness assumption of
Definition~\ref{def:pentagram-coefficient-relations}.
We impose the stronger normalization
\begin{equation}
A(e_0,0)=A(0,e_1)=1,
\label{eq:pure-power-normalization}
\end{equation}
which means that all nonconstant pure-power coefficients vanish.
For group-like series, this condition is equivalent to the
vanishing of the linear coefficients.

\begin{prop}\label{prop:five-cycle-shuffle}
Let $A\in\C\langle\!\langle e_0,e_1\rangle\!\rangle$
satisfy \eqref{eq:pure-power-normalization}, reversal
\eqref{eq:formal-reversal}, and the five-cycle equation
\eqref{eq:formal-five-factor-pentagon} in the completed residue
algebra \eqref{eq:five-residue-relations}. Then $A$ is group-like.
\end{prop}

\begin{proof}
Work in the completed tensor product
\[
\C\langle\!\langle x,y\rangle\!\rangle
\widehat\otimes
\C\langle\!\langle u,v\rangle\!\rangle.
\]
Thus $[x,u]=[x,v]=[y,u]=[y,v]=0$.
The substitution
\begin{equation}
(a,b,c,d,e)=(x+u,y,0,u,y+v)
\label{eq:shuffle-residue-specialization}
\end{equation}
respects the five disjoint-pair commutators in
\eqref{eq:five-residue-relations}. The remaining relation becomes
\[
[x,y]+[u,v]+[y+v,x+u]=0.
\]
The five-cycle therefore gives
\[
A(y,0)A(u,y+v)A(x+u,y)A(0,u)A(y+v,x+u)=1.
\]
Normalization and reversal reduce this to
\begin{equation}
A(x+u,y+v)=A(u,y+v)A(x+u,y).
\label{eq:shuffle-mixed-identity}
\end{equation}

Set $y=0$. The same identity gives
\begin{equation}
A(x+u,v)=A(u,v),\qquad [x,u]=[x,v]=0.
\label{eq:shuffle-central-shift}
\end{equation}
Consequently $A$ is unchanged on adding to its first argument
an element commuting with both arguments. Reversal gives
the analogous identity for the second argument:
\[
A(u,v+z)=A(v+z,u)^{-1}=A(v,u)^{-1}=A(u,v)
\quad\text{if }[z,u]=[z,v]=0.
\]
These identities apply to both factors on the right of
\eqref{eq:shuffle-mixed-identity}. Hence
\begin{equation}
A(x+u,y+v)=A(u,v)A(x,y)=A(x,y)A(u,v).
\label{eq:shuffle-tensor-identity}
\end{equation}
The left side is $\Delta A$ for the coproduct with primitive
generators, and the right side is $A\otimes A$. This proves
group-likeness. Equivalently, if $A_w$ denotes a word coefficient,
\[
A_rA_s=\sum_{w\in\operatorname{Sh}(r,s)}A_w,
\]
where $\operatorname{Sh}(r,s)$ counts the shuffles of $r,s$
with their multiplicities.
\end{proof}

\begin{cor}[Pentagram and shuffle]
\label{cor:pentagram-shuffle}
Let $A$ satisfy \eqref{eq:pure-power-normalization}, reversal,
and both pentagram coefficient equations. Then $A$ satisfies
shuffle and regularized double shuffle.
\end{cor}

\begin{proof}
Put $B_A(b,c)=\rev A(c,b)$. The endpoint calculation
\eqref{eq:formal-column-endpoint-general}, which uses reversal
but not group-likeness, gives
\[
S_A B_A(b,c)=K_\Phi\tau\sigma(S_A)=K_\Phi\tau(S_A).
\]
Iterating five times and using
$\prod_{i=0}^4\tau^i(K_\Phi)=1$ gives
\[
S_A\prod_{i=0}^4\tau^i\bigl(B_A(b,c)\bigr)=S_A.
\]
The upper entry of $S_A$ has constant term $1$, so it can
be cancelled. Word reversal preserves the defining ideal
of the residue algebra. Applying it to the resulting identity
and using $\rev B_A(s,t)=A(t,s)=A(s,t)^{-1}$ gives
\[
\bigl(A(b,c)A(d,e)A(a,b)A(c,d)A(e,a)\bigr)^{-1}=1.
\]
Proposition~\ref{prop:five-cycle-shuffle} now gives shuffle.
Having established group-likeness, apply
Corollary~\ref{cor:pentagram-double-shuffle}.
\end{proof}

\begin{rem}
The substitution \eqref{eq:shuffle-residue-specialization}
is the residue action on a tensor product of pullbacks
along two adjacent forgetful maps, $p_4$ and $p_3$.
On the first factor its five residues are $(x,y,0,0,y)$;
the formulas following \eqref{eq:five-residue-generators}
give $r_{i4}=0$. On the second factor they are $(u,0,0,u,v)$,
and $r_{i3}=0$. Adding the two actions gives
\eqref{eq:shuffle-residue-specialization}.
Thus \eqref{eq:shuffle-tensor-identity} expresses tensor
compatibility on these coefficient systems, in parallel
with the convolution specialization of the preceding subsection.
\end{rem}

\begin{rem}
If $[e_0e_1]A\ne0$, the lower pentagram equation alone
suffices in Corollary~\ref{cor:pentagram-shuffle}:
the lower entry of $S_A$ is then invertible, and its
endpoint identity gives the same five-cycle by cancellation.
\end{rem}

\appendix

\section{The second column of the Betti comparison}
\label{app:betti-comparison}

We prove the second-column formula in
Proposition~\ref{prop:comparison-betti-presentations}.
Keep $p=p_3$, $f=1-t$, and the displaced disk $\Sigma$ of
Subsection~\ref{subsec:two-dimensional-shrinking}. The coefficient
fiber $V_b=V_{q_c}$ identifies $S_\Sigma(V)$ with $V_b$;
transports inside the disk are independent of the path.
The fiber $V_f$ is used in the relative-cocycle presentation
of Subsection~\ref{subsec:projection-vanishing-cycles}:
\[
[c(z,q)]\longleftrightarrow(z,q)\in V_f\oplus V_f,
\qquad
E=V_f,\quad Q=V_f.
\]
In these coordinates,
\[
a(z)=(z,0),\qquad \rho(z,q)=q,\qquad
\can(z,q)=((T_*-1)z,q),
\]
where the last pair is read in the presentation
$\vartheta_\ell:E\oplus Q\xrightarrow{\sim}\Phi$.
We also retain
\[
\var\iota=a,\qquad \var\can=M_1-1,
\]
with $M_1$ the positive base monodromy at $t=1$.

The calculation has two parts. We first describe the fiber
traces of the displaced disk and compute their relative
cocycles with common coefficient paths. We then fold the
two sides of the base onto $I=(0,1]$. The correction away
from $t=1$ is a coboundary near $t=0$, while the normal cell
on $C_!$ determines the quotient coordinate in $Q$.
Together with the nearby trace, this determines the
vanishing-cycle class.

\subsection{Coefficient paths and fiber cocycles}

Choose small neighborhoods of the three singular values
and contractible complex bands around the remaining real
arcs. Over each band, a smooth trivialization of $p$
preserving its four marked sections lifts retraction
onto the real arc. Taper this retraction to the identity
near the singular neighborhoods and transport coefficients
through the homotopy. This preserves $\varphi_\Sigma$
and supplies fiber traces on the real arcs, while retaining
the pieces near singular fibers. The deformed representative
need not be an embedded disk.

All local models use the outward normal signs of
Subsection~\ref{subsec:two-dimensional-shrinking}.
Interpolation in these strict normal half-spaces avoids
$D_!$ and preserves the relative boundary in $D_*$,
so the supported class is unchanged.

On the smooth locus choose one horizontal connection
preserving the marked sections. Use it for both folds
and all coefficient transports. We prescribe it
along the fiber isotopy below and extend it to the
remaining bands. In every homotopy we retain the complete
prism chain, including its cutting and endpoint faces.
Internal faces then cancel by
$\partial\mathsf P+\mathsf P\partial=H_1-H_0$;
faces on $D_*$ are read as relative faces.

Take
\[
t_+=\frac32,\qquad t_0=\frac34.
\]
The corresponding fiber punctures are
\[
-2,0,1,\infty
\quad\text{and}\quad
0,1,4,\infty,
\]
respectively. We identify the fiber cohomology at $t_0$
with $\Psi$ by transport along the real interval towards
$t=1$.

In the fiber at $t_+$, let $A_+$ be the inner arc
$0\to1$, and let $B_+$ be the outer arc
\[
1\longrightarrow+\infty
\longrightarrow-\infty\longrightarrow0.
\]
The detours are those inherited from the displacement
of $\Sigma$: the detour at $\lambda=-2$ is upper in $x$,
and the detour at infinity is upper in $1/x$.
With the orientation
\[
\eta=
\frac{dt\wedge dx}{t(1-t)x(1-x)},
\]
the required fiber trace over $t>1$ is $B_+-A_+$.

The arc $A_+$ lies in
$P_{14325}$, whose two infinity
edges have the lower bypass. In its stated orientation,
$B_+$ passes successively through
\[
P_{13425}\ \xrightarrow{\ D_{13}\ }
P_{13524}\ \xrightarrow{\ D_{35}\ }\ P_{14235}.
\]
These are the intervals $(1,+\infty)$, $(-\infty,\lambda)$,
and $(\lambda,0)$, respectively. Their infinity edges have
the upper bypass.

These signs must be read in normal coordinates on the
compactification. For example, near $D_{13}\cap D_{24}$
put
\[
\nu=1/y,\qquad v=y/x=1-t,\qquad u=1/x=\nu v.
\]
Here $D_{13}=\{\nu=0\}$ and $D_{24}=\{v=0\}$.
The outward $D_{13}$ displacement has
$\operatorname{Im}\nu<0$. Its infinity bypass is therefore
upper in $u$ for $t>1$, and lower in $u$ for $0<t<1$.
At $D_{35}$, the outward displacement has
$\operatorname{Im}(y-1)<0$; since $y=(1-t)x$, its bypass
is upper in $x$ for $t>1$ and lower for $0<t<1$.
Consequently the three-chamber trace over $I$ is the
oriented outer arc
\[
0\longrightarrow-\infty\longrightarrow+\infty
\longrightarrow1,
\]
with a lower $x$-bypass at $\lambda$ and a lower
$u$-bypass at infinity. In particular, its infinity
detour lies in the upper $x$-half-plane.
In this orientation the three chambers are
\[
P_{13245}\ \xrightarrow{\ D_{13}\ }
P_{12453}\ \xrightarrow{\ D_{35}\ }\ P_{12435},
\]
corresponding to $(-\infty,0)$, $(\lambda,+\infty)$,
and $(1,\lambda)$, traversed with decreasing $x$.

Along $D_{25}$ the transition
$1/y=(1/x)(x/y)$ similarly changes the sign of the
$1/y$-normal on passing through $D_{13}$. It gives
the upper base bypass on both outer branches, $x<0$
and $x>1$, and the lower bypass for $0<x<1$.

Denote the lower and upper transports from $t_+$ to
$t_0$ by $\mathcal T_-$ and $\mathcal T_+$. With positive
base monodromy,
\[
\mathcal T_+=M_1\mathcal T_-.
\]
For the lower transport we use
\[
t_\theta=
\frac{3e^{i\theta}}{3e^{i\theta}-1},
\qquad
\lambda_\theta=1-3e^{i\theta},
\qquad 0\leq\theta\leq\pi.
\]
It is realized on the fibers by
\[
x\longmapsto
1+(x-1)e^{i\theta\chi(|x-1|)},
\]
where $\chi$ is zero for $r\leq2$ and $r\geq4$,
equals $1$ at $r=3$, and is linear on the intervening
intervals. This isotopy fixes $[0,1]$ and infinity.

In the target fiber choose the coefficient base point
$x=\frac12$. Use short paths to $0$ and $1$.
For the two other points use a shared stem from
$\frac12$ to $\frac12+6i$. The path to infinity continues
upwards; the path to $\lambda=4$ runs right to
$\frac{17}4+6i$, descends to $\frac{17}4+\frac14i$,
and approaches a small
peripheral circle about $4$ from its upper-right quadrant.
Take the infinity detours outside a circle containing
this finite part of the paths. These paths have cyclic
order $0,1,\lambda,\infty$ and give the meridian relation
\[
g_0g_1g_\lambda g_\infty=1.
\]

There are four coefficient transports
\[
h_I,h_A,h_B,h_Q:V_b\longrightarrow V_f.
\]
Their coefficients are read at the following intersections:
\[
\begin{array}{c|l}
h_I&
\text{the $g_1$-intersection of the class over }I,
\\
h_A&
\text{the $g_1$-intersection of }\mathcal T_-A_+,
\\
h_B&
\text{the $g_1$-intersection of }\mathcal T_-B_+,
\\
h_Q&
\text{the remaining intersection of }\mathcal T_-B_+
\text{ with the path to }\lambda.
\end{array}
\]
Transport coefficients first through $\Sigma$, then along
the stated fiber isotopy and coefficient path. The cocycles
below include the intersection signs.

For $b\in V_b$, let $\xi_{I,b}$ be the fiber class of
$\Sigma$ over $t_0$, and put
\[
\xi_b=\mathcal T_-B_+(b).
\]
Before restoring the gauge $w_\lambda=0$, the relevant
entries are the following. The $z_1$ column records the
signed meridian intersections; the other two columns
record the relative terms on the chosen peripheral paths.
Thus the quotient coordinate is $w_\infty-w_\lambda$.
\[
\begin{array}{c|ccc}
&z_1&w_\lambda&w_\infty\\ \hline
\xi_{I,b}&h_Ib&0&-h_Qb\\
\xi_b&-h_Bb&h_Qb&0.
\end{array}
\]
For $\xi_{I,b}$, the path to $\lambda$ misses the support,
the path to infinity meets its upper detour once negatively,
and the short meridian at $1$ gives $+h_Ib$. For $\xi_b$,
that meridian gives $-h_Bb$, and the descending branch
towards $\lambda$ meets the transported arc once positively.
The two further intersections on the shared stem bound
a region without punctures. Their opposite signs and
homotopic coefficient paths make their contributions cancel.

The equality of the peripheral coefficients in the two
rows uses continuation along $D_{24}$. At the mixed node
put $u=1/x$, so $f=uy$, $C_!=\{u=0\}$ and
$C_*=\{y=0\}$. On the two branches of the cut disk write
$y=\delta r$, where $\delta=\pm1$ and $r\geq0$.
Choose a positive function $k(r)$ with $k(0)>0$ and
$k(r)=r$ outside a small node neighborhood. The branchwise
displacements
\begin{equation}
F_\delta(u,r)=
\bigl(u+i\epsilon\delta k(r),\,\delta r\bigr)
\label{eq:mixed-node-displacement}
\end{equation}
give
\[
f\circ F_\delta=\delta ru+i\epsilon r k(r).
\]
Both branches use the upper $f$-half-plane, hence the lower
$t$-half-plane. Their distinct node endpoints
$(i\epsilon\delta k(0),0)\in D_*\setminus D_!$
avoid the removed node.

We check the coefficient paths themselves in coordinates
$(f,y)$, where the fiber punctures are $0,f,1,\infty$.
The surviving intersection with the path to $\lambda$
comes from a point $x=1+r$ of $B_+$, with $3<r<4$.
Put $\mu=4-r$. Under the specified lower transport,
\[
f_\theta=\frac1{1-3e^{i\theta}},\qquad
y_\theta=\frac{1+r e^{i\mu\theta}}{1-3e^{i\theta}}.
\]
For $0<\theta<\pi$, the imaginary numerator of the
second expression after rationalization is
\[
3\sin\theta+r\sin(\mu\theta)
 +3r\sin((1-\mu)\theta)>0.
\]
Moreover,
\[
|f_\theta|\leq\tfrac12,\qquad
|y_\theta|\geq\frac{r-1}{4}>\tfrac12.
\]
Choose $\rho_0$ strictly between $1/2$ and $(r-1)/4$.
After a common initial path in $\Sigma$, compare the
two coefficient paths from $f=-1/2$, $y>1$ to the top
of the shared stem. The disk continuation uses the upper
$f$ bank. The $B_+$ continuation first passes through
the $D_{13}$ bypass, which is upper in $y$, and then
follows the displayed isotopy. Both paths can therefore
be taken in the product of
\[
\{0<|f|<\rho_0,\ \operatorname{Im}f\geq0\}
\quad\text{and}\quad
\{\operatorname{Im}y>0,\ |y|>\rho_0\},
\]
with small neighborhoods added at their real endpoints.
The factors are contractible and the product avoids
$y=0,f,1$ and $f=0,1$. The paths are therefore homotopic
in $U$ relative to their endpoints, and their final paths
down the shared stem coincide. Thus the same transport
$h_Q$ occurs in both rows for arbitrary $V$.

Consequently, the relative cocycles are
\begin{equation}
\label{eq:comparison-fiber-cocycles}
\begin{aligned}
\xi_{I,b}
 &=c(h_Ib,-h_Qb),\\
\xi_b
 &=c\bigl(-h_Bb-(T_1-1)h_Qb,-h_Qb\bigr),\\
\mathcal T_-A_+(b)
 &=a(h_Ab).
\end{aligned}
\end{equation}
Here $T_1$ is the fiber monodromy about $x=1$.
The term $(T_1-1)h_Qb$ comes from restoring the gauge
$w_\lambda=0$.

\subsection{Folding the supported class}

\begin{proof}[Proof of Proposition~\ref{prop:comparison-betti-presentations}]
The cut constructed below separates the upper and lower
folds. We keep its endpoint faces until their compact-support
correction has been computed, and then read the nearby
trace and the relative normal cell at $t=1$.

We first join the two folds at infinity. Use the base
coordinate
\[
w=\frac{t}{1-t}=x\nu-1,\qquad \nu=1/y.
\]
Near $D_{14}\cap D_{25}$ the displacement is
\[
(x,\nu)=(s_1+i\epsilon,s_2-i\epsilon),\qquad
w=s_1s_2+\epsilon^2-1+i\epsilon(s_2-s_1).
\]
The cut $s_1=s_2$ separates the upper and lower
base half-planes. On it,
\[
w=s_1^2+\epsilon^2-1\in(-1,0)
\]
for a sufficiently small node neighborhood.

At the cut endpoints $s_1=\pm a$, put
$\lambda_c=a^2+\epsilon^2<1$. The fiberwise rotation
\[
x=\sqrt{\lambda_c}e^{i\theta},\qquad
\nu=\sqrt{\lambda_c}e^{-i\theta}
\]
joins them to real points by taking $\theta$ to $0$
on the positive branch and to $\pi$ on the negative
branch. It keeps $x\nu=\lambda_c$ fixed. Continue the
cut along
\[
x=\pm\sqrt\lambda,\qquad y=\pm\lambda^{-1/2},
\qquad \lambda_c\leq\lambda\leq1.
\]
Its two ends lie on $D_{12}$ and $D_{45}$ over $t=0$.
The prescribed fiber rotations extend to a thin strip
by relative isotopy, since they avoid $D_!$ and keep
the endpoints on $D_*$. We use this further deformation
of the representative. The cut is then the image of
a proper arc in the disk. It separates
the two pieces whose traces over $t>1$ are $B_+$
and $A_+$, respectively; the bypass signs calculated
above specify the upper fold on the first and the
lower fold on the second.

Away from the mixed-fiber neighborhood, the folds
have the explicit base formulas
\[
H_s^\pm(w)=|w|\exp\bigl(i(1-s)\arg_\pm w\bigr),
\qquad 0\leq s\leq1,
\]
where the arguments approach $\pm\pi$ on the two banks
of the negative axis. They end on the positive
$w$-axis, hence on $I$, and are lifted with the common
horizontal connection fixed above. The infinity collar
lies in the corresponding open half-planes, so these
folds avoid $w=-1$. For $t>1$ we have $|w|>1$.
The loop comparing the upper and lower folds is the
positive meridian about $t=1$, giving
$\mathcal T_+=M_1\mathcal T_-$.
The piece in the mixed-fiber neighborhood is retained;
its normal cell is calculated below.

The straightened real strips over $t<0$ fold into
$0<t<\frac12$. The displaced pieces in the infinity
neighborhood need not satisfy this precise bound.
Their $w$-coordinates remain bounded, however, since
that fixed neighborhood avoids the mixed fiber, where
$w$ has its pole. Their folds therefore end on a
bounded segment of the positive $w$-axis. All these
pieces contribute no germ at $t=1$; their cutting and
endpoint faces are retained in the completed correction.

On the cut, $|w|<1$, and the two folds bound the disk
\[
w\longmapsto\zeta w,\quad |\zeta|\leq1,
\qquad
t\longmapsto\frac{\zeta t}{1-t+\zeta t}.
\]
For the fixed nonzero displacement, the cut and its
continuations satisfy $|w|\leq1-\epsilon^2$.
This bounds the cut filling; it does not bound every
displaced piece. After folding, the other retained
pieces lie on a bounded positive real segment $[0,R_c]$.
Choose $\rho_c$ with $1-\epsilon^2<\rho_c<1$ and let $B_0$
be a relatively compact contractible neighborhood of
\[
\{|w|\leq1-\epsilon^2\}\ \cup\ [0,R_c],
\]
obtained by adjoining a thin neighborhood of $[0,R_c]$
to $\{|w|<\rho_c\}$. It can be chosen to avoid $w=-1$.
The completed correction, including its endpoint faces,
has compact support in $B_0$. This neighborhood contains
$t=0$ and avoids $t=1,\infty$.

We calculate its endpoint faces. Put
$\Xi=\Psi_t(\mathcal F)$, the nearby-cycle space at
$t=0$, and $D_0=M_0-1$. On $B_0$, $\mathcal F$ is a
derived $*$-extension. Compute compact support using a
closed disk relative to its boundary circle. Before the
perverse shift, restriction from the disk to its boundary
identifies both cochain complexes with the circle complex
$[\Xi\xrightarrow{D_0}\Xi]$. The compact-support complex
is therefore the cone of this identity, with the perverse
shift taken into account. Ordering its degree-zero terms
as the cut face and the stalk term gives
\[
\Xi\xrightarrow{\ h\mapsto(D_0h,h)\ }
\Xi\oplus\Xi
\xrightarrow{\ (z,h)\mapsto z-D_0h\ }\Xi.
\]
The terms have degrees $-1,0,1$. This complex is contractible:
a contracting homotopy is $(z,h)\mapsto h$ in degree zero
and $k\mapsto(k,0)$ in degree one.
Read the coefficient $h$ on the upper bank and orient
the angular jump from the upper bank to the lower bank.
For $|w|<1$ the lower-bank coefficient is $M_0h$;
equivalently, the upper transport is $M_0^{-1}$ times
the lower transport. The angular endpoint difference
is therefore $D_0h$. The other cutting faces cancel
as faces of the common prism chain. Consequently the
actual residual term is
\begin{equation}
(D_0h,h)=d(h).
\label{eq:zero-collar-coboundary}
\end{equation}
The same holds for sums of strips with transported
coefficients. It removes the correction, including its
monodromy-invariant part, without adding a first-summand term.

Folding reverses the base orientation on $t>1$, so its
trace $B_+-A_+$ contributes $-M_1\xi_b+a(h_Ab)$.
It follows that near the right end of $I$ the folded
nearby-cycle class is
\[
\xi_{I,b}-M_1\xi_b+a(h_Ab).
\]
Put
\[
e_b=h_Ib+h_Bb+(T_1-1)h_Qb.
\]
Equation~\eqref{eq:comparison-fiber-cocycles} gives
$\xi_{I,b}-\xi_b=a(e_b)$, so the preceding expression is
\begin{equation}
\label{eq:comparison-nearby-trace}
a(e_b+h_Ab)-(M_1-1)\xi_b.
\end{equation}

We also compute the relative term on $C_!=D_{24}$.
The orientation near the mixed node is
\[
\eta=\frac{du\wedge dy}{y(1-uy)(1-u)}
     =\frac{df\wedge dy}{y^2(1-f)(1-u)}.
\]
Both branches of \eqref{eq:mixed-node-displacement}
have the same product orientation in $(f,y)$;
passing to $t=1-f$ gives the negative normal cell.
Cut the normal annulus along the ray $f>0$ and read
coefficients on the bank determined by the lower
continuation. The two banks have coefficients
$\rho\xi_b$ and $T_!\rho\xi_b$, where $T_!$ is normal
monodromy on the component term. The signed relative
cell has coefficient $-\rho\xi_b$ and boundary
\[
(1-T_!)\rho\xi_b.
\]
Since the original extension is $!$ along this component,
the Milnor map $\can$ identifies its normal
nearby and vanishing terms. The quotient coordinate of
the folded class is consequently
\begin{equation}
\label{eq:comparison-quotient}
-\rho\xi_b=h_Qb.
\end{equation}
This records the relative cell itself, including its
monodromy-invariant part.

In the pushout complex, a negative normal cell with
nearby coefficient $\xi=c(z,q)$ is explicitly
\[
[(-\xi,0)]=-\can\xi=(-\delta_*z,-q).
\]
The term $-\delta_*z$ already records its face on the
$*$-component. The separated node endpoints of
\eqref{eq:mixed-node-displacement} lie in the relative
boundary and introduce no further face.

The map
\[
(\var,\pi):
\Phi\longrightarrow\Psi\oplus Q
\]
is injective: its kernel lies in $\iota E$, and
$\var\iota=a$ is injective. Since $\var\can=M_1-1$
and $\pi\can=\rho$, equations~\eqref{eq:comparison-nearby-trace}
and~\eqref{eq:comparison-quotient} therefore give
\[
\varphi_I^{-1}\varphi_\Sigma(b)
=
\iota(e_b+h_Ab)-\can\xi_b.
\]
Finally, applying
$\can(z,q)=((T_*-1)z,q)$ to
\eqref{eq:comparison-fiber-cocycles} yields
\[
\vartheta_\ell^{-1}\varphi_I^{-1}\varphi_\Sigma(b)
=
\left(
\bigl[h_I+h_A+T_*h_B+T_*(T_1-1)h_Q\bigr]b,\,
h_Qb
\right).
\]
Together with the main-pentagon identification, this
proves the proposition.
\end{proof}

\section{Logarithmic calculations}
\label{app:logarithmic-calculations}

\subsection{The logarithmic direct image}
\label{app:logarithmic-direct-image}

We verify \eqref{eq:explicit-outer-gluing} and
\eqref{eq:explicit-component-gluing} in the conventions
of Subsection~\ref{subsec:associator-formula}.
In this calculation $V$ denotes the graded de Rham
coefficient space of the universal coefficient object.

\begin{proof}
Use the coordinates $(0,\infty,1,x,y)$ and write the
universal connection as $\nabla=d-\Omega$, where
\[
\Omega=r_{14}\,d\log x+r_{34}\,d\log(x-1)
 +r_{15}\,d\log y+r_{35}\,d\log(y-1)
 +r_{45}\,d\log(x-y).
\]
Substitute $y=(1-t)x$ and $\lambda=(1-t)^{-1}$.
For this calculation put
\[
\begin{gathered}
\alpha=r_{23},\quad\beta=r_{34},\quad\gamma=r_{35},\\
s=\alpha+\beta+\gamma=-r_{13},\quad
\delta=r_{15}+r_{35},\quad\eta=r_{45}.
\end{gathered}
\]
Then
\[
\Omega=\alpha\,d\log x+\beta\,d\log(x-1)
 +\gamma\,d\log(x-\lambda)
 +\delta\,d\log(1-t)+\eta\,d\log t.
\]
We use signed peripheral coordinates: $w_\lambda,w_\infty$
are the negatives of the local primitive coordinates
in the usual mapping fiber. With the common sign of
the degree-zero term fixed accordingly, the global
relative logarithmic complex is
\begin{equation}
C^0=V,\qquad
C^1=V(-1)^{\oplus3}\oplus V^{\oplus2},\qquad
C^2=V(-1)^{\oplus2},
\label{eq:relative-logarithmic-fiber-complex}
\end{equation}
with differentials
\[
\begin{aligned}
d^0(v)&=(\alpha v,\beta v,\gamma v;v,v),\\
d^1(z_0,z_1,z_\lambda;w_\lambda,w_\infty)
 &=\bigl(z_\lambda-\gamma w_\lambda,\,
 z_0+z_1+z_\lambda-s w_\infty\bigr).
\end{aligned}
\]
Each residue has Tate degree one, so both differentials
preserve the total grading.
Every class has a unique representative with
$w_\lambda=0$, namely
\[
(sq-z,z,0;0,q).
\]
Consequently its first cohomology is $V(-1)\oplus V$,
with
\[
F^pH^1=F^{p-1}V\oplus F^pV,\qquad
W_mH^1=W_{m-2}V\oplus W_mV.
\]
The corresponding differential-form representative is
\[
\omega(z,q)=(sq-z)\frac{dx}{x}
       +z\frac{dx}{x-1},\qquad w_\lambda=0,
\quad w_\infty=q.
\]
Put $B_\lambda=\eta/(\lambda-1)-(\delta+\eta)/\lambda$.
The moving boundary $x=\lambda$ contributes the value of
the fiber form on $\partial_x$. Differentiating the cocycle
therefore produces the boundary value
\[
w_\lambda'=\omega(z,q)(\partial_x)|_{x=\lambda}
 =\frac{sq}{\lambda}+\frac{z}{\lambda(\lambda-1)}.
\]
Put
\[
B_x=\frac{\alpha}{x}+\frac{\beta}{x-1}
          +\frac{\gamma}{x-\lambda},\qquad
h=\frac{sq}{\lambda}+\frac{z}{\lambda(\lambda-1)}.
\]
Restoring $w_\lambda=0$ shows that the coefficient of
$\Omega_{\mathrm{GM}}$ is represented by the form
\[
B_\lambda\omega-\frac{\gamma\omega}{x-\lambda}+B_xh\,dx.
\]
Its pole at $x=\lambda$ cancels. The residue at $x=1$
and the retained peripheral entry at $\infty$ give
\[
z'=\left(B_\lambda+\frac{\gamma}{\lambda-1}\right)z+\beta h,
\qquad q'=B_\lambda q+h.
\]
Compatibility with the residue at $\infty$ uses
$[s,B_\lambda]=0$, which follows from
$[r_{13},r_{45}]=[r_{13},r_{15}+r_{35}]=0$.
Thus the Gauss--Manin connection is
$d-\Omega_{\mathrm{GM}}$, with
\[
\Omega_{\mathrm{GM}}=
\begin{pmatrix}
B_\lambda+\dfrac{\gamma}{\lambda-1}
 +\dfrac{\beta}{\lambda(\lambda-1)}
 &\dfrac{\beta s}{\lambda}\\[6pt]
\dfrac1{\lambda(\lambda-1)}&B_\lambda+\dfrac{s}{\lambda}
\end{pmatrix}d\lambda.
\]
The residues of $\Omega_{\mathrm{GM}}$ at $t=0,1$ are $X,Y$ in
\eqref{eq:explicit-outer-gluing}. They have total Tate degree
one, so the limiting weight filtration is unchanged; the
Hodge filtration extends in the trivial logarithmic bundle.
Thus nearby cycles have the same displayed filtrations.

For vanishing cycles, apply logarithmic localization
to the same relative complex. In logarithmic variation
coordinates, the normal map $\can$ on $D_{15}$ is
multiplication by its residue $r_{15}$, so the logarithmic
version of the pushout of
Subsection~\ref{subsec:projection-vanishing-cycles} is the cokernel of
\[
V(-1)\longrightarrow
  [V(-1)\oplus V]\oplus V(-2),\qquad
u\longmapsto((u,0),-eu).
\]
The quotient map $[(z,q,v)]\mapsto(v+ez,q)$ is graded
and identifies $\Phi_{\mathrm{dR}}$ with $V(-2)\oplus V$
preserving $F$ and $W$. In these coordinates
$\can=c_{\mathrm{out}}=\operatorname{diag}(r_{15},1)$.
Logarithmic variation sends the first summand to the
interval class, so $v_{\mathrm{out}}J=\binom10$.
Together with $v_{\mathrm{out}}c_{\mathrm{out}}=Y$,
this gives the displayed $v_{\mathrm{out}}$.

Finally, taking the normal limit with
$f=y/x$ on $D_{15}$ gives
\[
\Omega=r_{15}\,d\log f+r_{23}\,d\log x
             +r_{34}\,d\log(x-1),
\]
which proves \eqref{eq:explicit-component-gluing}.
\end{proof}

\subsection{Formal endpoint and frame identities}
\label{app:formal-endpoint-identity}

Let $A$ be a formal series with constant term $1$ satisfying
\eqref{eq:formal-reversal}; group-likeness is not assumed.
First work in the free completed
algebra $\mathbb C\langle\!\langle b,c\rangle\!\rangle$,
before substituting the residue generators. Put
\[
T=\rev A(b,c)=1+bT_b+cT_c,\qquad
T^{-1}=1+bU_b+cU_c.
\]
Put $B_A(b,c)=\rev A(c,b)$. Reversal gives
\[
T^{-1}=B_A(b,c).
\]
Projecting $TT^{-1}=1$ onto words beginning with $b$
gives $U_b+T_bT^{-1}=0$. Hence
\[
(1+cT_c)T^{-1}
 =(T-bT_b)T^{-1}=1+bU_b.
\]
The rightmost-letter decomposition of $A$ gives
\[
T_c=\mathscr F_A(b,c),\qquad U_b=\mathscr F_A(c,b).
\]
Consequently
\begin{equation}
\bigl(1+c\mathscr F_A(b,c)\bigr)B_A(b,c)
 =1+b\mathscr F_A(c,b).
\label{eq:formal-component-endpoint}
\end{equation}
This formal endpoint semi-holonomy identity uses only reversal.
If $A$ is group-like, the antipode identity further gives
$B_A(b,c)=A(-b,-c)$, recovering the endpoint formula used
in Subsection~\ref{subsec:pentagram-to-pentagon}.

After substitution of the residue generators, the rotation
and the reflection fixing $C_0$ act by
\[
\begin{aligned}
\tau(a,b,c,d,e)&=(c,d,e,a,b),\\
\varrho(a,b,c,d,e)&=(d,c,b,a,e),\qquad
\varrho=\tau\sigma.
\end{aligned}
\]
Put
\[
K_\Psi=\begin{pmatrix}1&a-d\\0&1\end{pmatrix},
\qquad
K_\Phi=\begin{pmatrix}1&e(a-d)\\0&1\end{pmatrix}.
\]
Direct substitution in
\eqref{eq:explicit-outer-gluing} gives
\begin{equation}
\begin{gathered}
K_\Psi\varrho(X)=XK_\Psi,\qquad
K_\Psi\varrho(Y)=YK_\Psi,\\
K_\Phi\varrho(c_{\mathrm{out}})
 =c_{\mathrm{out}}K_\Psi,\qquad
K_\Psi\varrho(v_{\mathrm{out}})
 =v_{\mathrm{out}}K_\Phi.
\end{gathered}
\label{eq:formal-side-intertwiners}
\end{equation}
For the second identity, the difference of the
upper-right entries is the negative of the cyclic
commutator relation in
\eqref{eq:five-residue-relations}; the other identities
follow from the same relation and the disjoint-pair
commutators.

Write the raw outer column as
\[
P_A=\begin{pmatrix}
1+e\mathscr F_A(X,Y)_{11}\\
\mathscr F_A(X,Y)_{21}
\end{pmatrix},\qquad
S_A=P_A\bigl(1+c\mathscr F_A(b,c)\bigr).
\]
The intertwining identities and $K_\Phi J=J$ imply,
coefficient by coefficient,
\[
P_A=K_\Phi\varrho(P_A).
\]
Since $\varrho$ exchanges $b,c$, the endpoint identity
\eqref{eq:formal-component-endpoint} now gives
\begin{equation}
S_A B_A(b,c)=K_\Phi\varrho(S_A).
\label{eq:formal-column-endpoint-general}
\end{equation}
For group-like $A$ this specializes to
\begin{equation}
S_A A(-b,-c)=K_\Phi\varrho(S_A).
\label{eq:formal-column-endpoint}
\end{equation}
If $A$ satisfies the pentagram coefficient relations,
then $\sigma(S_A)=S_A$, and consequently
$\varrho(S_A)=\tau(S_A)$.

Finally, the five matrices $\tau^i(K_\Phi)$ are upper
unipotent, and the sum of their upper-right entries is
\[
\begin{aligned}
&e(a-d)+b(c-a)+d(e-c)+a(b-e)+c(d-b)\\
&\qquad=[a,b]+[b,c]+[c,d]+[d,e]+[e,a]=0.
\end{aligned}
\]
Thus $\prod_{i=0}^4\tau^i(K_\Phi)=1$.

\subsection{The reflection change of frames}
\label{app:reflection-frames}

We prove Lemma~\ref{lem:universal-graded-change} by bounding
transition degrees independently of $A$, then determining
the entries from the first period terms.

\begin{lemma}
\label{lem:graded-frame-bound}
The universal logarithmic frames in
\eqref{eq:explicit-outer-gluing} identify the graded
de Rham space of vanishing cycles with
$V_{\mathrm{dR}}(-2)\oplus V_{\mathrm{dR}}$,
naturally in the coefficient object. In these frames,
the Hodge-compatible change between two projection
presentations, transported to cohomology by $\varphi_I^F$, has the form
\[
\begin{pmatrix}d_1&B\\0&d_2\end{pmatrix},
\qquad d_1,d_2\in\mathbb C^*,\qquad
B\in U(\mathfrak p_5)_2.
\]
Here $\mathfrak p_5$ is the complex Lie algebra with
presentation \eqref{eq:five-residue-relations}, with
generators of degree one. Under relabelling, the
coefficient action is first transported accordingly.
At the fixed corner the source transition is
$C_\sigma^W=\varepsilon\,1$, where $\varepsilon=-1$
is the orientation sign specified above.
\end{lemma}

\begin{proof}
Use the total Tate grading, including the shifts on
the component terms. The relative complex
\eqref{eq:relative-logarithmic-fiber-complex} and its
gauge in Appendix~\ref{app:logarithmic-direct-image} give
\[
\Psi_{\mathrm{dR}}=V(-1)\oplus V,\qquad
\Phi_{\mathrm{dR}}=V(-2)\oplus V
\]
with both the Hodge and weight filtrations. In particular,
\[
F^p\Phi_{\mathrm{dR}}=F^{p-2}V\oplus F^pV,\qquad
W_m\Phi_{\mathrm{dR}}=W_{m-4}V\oplus W_mV.
\]
Thus these are homogeneous coordinates for the Hodge
grading. The identity in the lower-left entry of $X$
has total Tate degree one because the two nearby
coordinates have different shifts.

The finite regular modules used below have rational Hodge realizations.
Fix an ordinary base point $x_\circ\in U(\Q)$ and take the duals of the
length-$\leq N$ homotopy-invariant iterated integrals from $x_\circ$.
These are finite quotients of the canonical unipotent path variation
of \cite[Sections~6--7]{CarlsonHain}, and are admissible rational
variations by \cite{HainZucker}.
The logarithmic bar model for $\M{5}$ identifies their filtered
de Rham realizations with
\[
 R_N=U(\mathfrak p_5)/U(\mathfrak p_5)_{>N},
 \qquad \nabla=d-\Omega.
\]
Here $\Omega$ is the universal connection of
Appendix~\ref{app:logarithmic-direct-image}; its residues act by left
multiplication. The filtrations are
\[
W_{-2m}R_N=\bigoplus_{k\ge m}(R_N)_k,\qquad
F^pR_N=\bigoplus_{k\le-p}(R_N)_k,
\]
where $0\le k\le N$. Degree $k$ has Tate type $\Q(k)$,
dual to bar length $k$ of type $(k,k)$; see
\cite[Sections~6.1--6.3]{BrownModuli}.
The forms $d\log x,d\log(x-1),d\log y,d\log(y-1),d\log(x-y)$
use functions with simple boundary zeros and poles. Thus these
quotients are coefficient objects of Subsection~\ref{subsec:iterated-kummer}.
Their corner fibers are iterated unipotent limits; admissibility
provides the rational limiting structures, and the logarithmic
frames are regularized as in \cite[Theorem~6.12]{BrownModuli}.
These truncations separate the completed algebra. For a
homogeneous element $h$ of degree $k$, right
multiplication is a horizontal map
$R_N\to R_N(-k)$ preserving $F$ and $W$.
It acts on the filtered logarithmic complexes for direct
images and nearby cycles, preserving both filtrations on
cohomology. Shrinking is natural for the underlying complex
local-system map, and $s^F$ is natural for maps preserving $F,W$.
Thus $\varphi_I^F$ and the transition commute with these
maps, Tate twists, and truncation. Each matrix entry is
therefore left multiplication by an element of
$\widehat U(\mathfrak p_5)$, determined by its value at
$1$. This uses the complex de Rham realization, without
requiring right multiplication to be rational as a Hodge
morphism.

The transition is obtained by conjugating the cohomological
relabelling map by the Hodge-compatible maps $\varphi_I^F$.
It therefore preserves both $F$ and $W$, and hence the
Hodge grading. Preservation of the total grading forces the entry
degrees to be
\[
\begin{pmatrix}0&2\\-2&0\end{pmatrix}.
\]
There are no negative degrees, and degree zero consists
of scalars. Invertibility makes both diagonal scalars
nonzero. The same argument for the corner source
$V_{\mathrm{dR}}(-2)$ gives a degree-zero scalar transition.
In the chart $(0,\infty,1,x,y)$ the reflection is
$\sigma(x,y)=(1-y,1-x)$. At $q_0=(1,0)$ it exchanges
the inward normal coordinates $1-x,y$ without rescaling.
Thus on the augmentation it is the identity on coefficients
times the normal-orientation sign $\varepsilon$.
This proves the source assertion as well.
\end{proof}

\begin{proof}[Proof of Lemma~\ref{lem:universal-graded-change}]
The geometric Theorem~\ref{thm:pentagram-geometric} gives
the framed identity \eqref{eq:framed-pentagram-relation}.
Retain the Hodge convention
\[
G_\nu=R_\nu(A(-e_0,-e_1)_Y)U(c_\nu v_\nu),
\qquad \widetilde A=A(-e_0,-e_1).
\]
Since
$(c_{\mathrm{out}}v_{\mathrm{out}})J=Je$, we have
\[
U(c_{\mathrm{out}}v_{\mathrm{out}})J=JU(e).
\]
The residue $e$ commutes with $b,c$. Hence the full
Hodge column is
\[
S_{\widetilde A}\,U(e)U(c).
\]
The right factor is invertible and invariant under
$\sigma$, because $ec=ce$. It cancels from
\eqref{eq:framed-pentagram-relation}.

By Lemma~\ref{lem:graded-frame-bound}, the remaining
target transition is upper triangular, with scalar
diagonal and an off-diagonal entry of residue degree two.
At zero residues the column is
\[
\binom{1}{-\zeta(2)}.
\]
Both entries are nonzero. Comparing with the reflected
column fixes the diagonal entries as the corner-source
orientation sign. After cancelling that common sign,
\eqref{eq:framed-pentagram-relation} therefore reads
\[
S_{\widetilde A}
 =\begin{pmatrix}1&B_\sigma\\0&1\end{pmatrix}
   \sigma S_{\widetilde A},
\qquad B_\sigma\in U(\mathfrak p_5)_2.
\]

The first nonconstant upper term is
\[
S_{\widetilde A}^{(1)}
 =1-\zeta(2)(ea+cb)+O(\deg\ge3),
\]
whereas $S_{\widetilde A}^{(2)}=-\zeta(2)+O(\deg\ge1)$.
The quadratic expression $ea+cb$ is fixed by $\sigma$.
Taking residue degree two in the upper equation gives
\[
0=-\zeta(2)B_\sigma.
\]
Thus $B_\sigma=0$. Both frame transitions are therefore
the scalar orientation sign $\varepsilon$, as asserted.

\end{proof}

\section{Coefficient equations}
\label{app:coefficient-calculations}

\subsection{Coefficient extraction and gradings}

We calculate the two block equations separately, using the
residue algebra, $X,Y$, and
$\sigma(a,b,c,d,e)=(b,a,e,d,c)$ from
\eqref{eq:explicit-pentagram-column}.
Write $A=\sum_w A_w w$. For a word $w=ue_1$, put
\[
\begin{aligned}
L_w^{(1)}&=e\,\rev(u)(X,Y)_{11},\qquad
L_w^{(2)}=\rev(u)(X,Y)_{21},\\
C_w&=c\,\rev(u)(b,c).
\end{aligned}
\]
Expanding the two factors in the column gives
\begin{equation}
\begin{aligned}
S_A^{(1)}
 &=1+\sum_w A_w\bigl(L_w^{(1)}+C_w\bigr)
       +\sum_{u,v}A_uA_vL_u^{(1)}C_v,\\
S_A^{(2)}
 &=\sum_w A_wL_w^{(2)}
       +\sum_{u,v}A_uA_vL_u^{(2)}C_v.
\end{aligned}
\label{eq:mzv-coefficient-extraction}
\end{equation}
All indices are words ending in $e_1$. Reduce residue words
modulo \eqref{eq:five-residue-relations} and equate coefficients
of $S_A^{(k)}-\sigma S_A^{(k)}$ to zero for $k=1,2$.
The equations are rational linear combinations of $A_w$
and $A_uA_v$. Anti-invariance refers to $\sigma$, not
to parity of word length.

For the period associator, a convergent word has coefficient
\[
A_{e_0^{n_1-1}e_1\cdots e_0^{n_k-1}e_1}
 =(-1)^k\zeta(n_1,\ldots,n_k),\qquad n_1\ge2.
\]
The other coefficients are determined by shuffle
regularization. We write
$z_{n_1\cdots n_k}=\zeta(n_1,\ldots,n_k)$.
Associator weight is word length, with additive weight
on products. It differs from residue degree in the
lower block:
\[
\begin{array}{c|c|c}
\text{block}&\text{associator weight}&\text{residue degree}\\\hline
S_A^{(1)}\ \text{(upper)}&n&n\\
S_A^{(2)}\ \text{(lower)}&n&n-2
\end{array}
\]
Thus the constant lower term has weight two. We denote
the part of associator weight $n$ by $S_{A,[n]}^{(k)}$;
the corresponding kernel
$\mathscr F_{A,[n-1]}(u,v)$ has word degree $n-1$.

For the coefficient reductions below, order residue
words lexicographically with $a<b<c<d<e$. In each degree,
row-reduce all left and right multiples of the six
quadratic relations in \eqref{eq:five-residue-relations}.
The nonpivot words give the quotient coordinates used
in the tables. The quotient dimensions in residue
degrees $2,3,4,5$ are $19,65,211,665$, respectively.
Convergent MZVs and their products remain separate symbols,
except where a lower-weight relation is explicitly used.
We use shuffle only to regularize the remaining coefficients;
the displayed ranks are not taken modulo all shuffle and
reversal relations.

\subsection{Weights two and three}

The first upper terms are
\[
\begin{aligned}
S_A^{(1)}={}&1-z_2(ea+cb)-z_3(ea^2+cb^2)\\
&+z_{21}\bigl(e^2a+c^2b-2eae-2cbc
                  +ec(d-a-b)\bigr)+\cdots,
\end{aligned}
\]
whereas $S_{A,[2]}^{(2)}=-z_2$.
The displayed upper expressions are invariant under
$\sigma$, using $ec=ce$, as is the lower weight-two term.
Thus neither block gives a weight-two equation; the upper
block gives none in weight three, even without $z_{21}=z_3$.

For the lower weight-three term, shuffle regularization
gives
\[
\mathscr F_{A,[2]}(u,v)
 =-z_3u^2+z_{21}(vu-2uv).
\]
Substitution in \eqref{eq:mzv-coefficient-extraction}
yields
\[
S_{A,[3]}^{(2)}=-z_3(a+d)-z_{21}(b+c+e).
\]
Its reflection difference is $(z_{21}-z_3)(a-b)$.
Since $a-b\ne0$ in residue degree one, the lower block
gives the Euler identity
\begin{equation}
\zeta(2,1)=\zeta(3).
\label{eq:pentagram-weight-three}
\end{equation}

\subsection{Weight four}

Put
\[
z_4=\zeta(4),\quad z_{31}=\zeta(3,1),\quad
z_{22}=\zeta(2,2),\quad z_{211}=\zeta(2,1,1),
\qquad p=\zeta(2)^2.
\]
Shuffle regularization gives
\begin{equation}
\begin{split}
\mathscr F_{A,[3]}(u,v)={}&-z_4u^3+z_{31}vu^2
 +z_{22}uvu-(z_{22}+2z_{31})u^2v\\
 &-z_{211}v^2u+3z_{211}vuv-3z_{211}uv^2.
\end{split}
\label{eq:unreduced-weight-four-kernel}
\end{equation}
The mixed products in the upper and lower blocks are
$p\,eacb$ and $p\,cb$, respectively.

The upper block has residue degree four. Its reduction
uses all $450$ relation multiples of total degree four;
their rank is $414$. The lower block has residue degree
two. The following coefficients of the two reflection
differences span their respective equations:
\[
\begin{array}{c|c|l}
\text{block}&\text{residue word}&\text{coefficient}\\\hline
S_A^{(1)}&ecaa&z_{211}-z_{31}-z_{22}\\
S_A^{(1)}&ecae&p+z_{31}-z_{22}-2z_{211}\\\hline
S_A^{(2)}&a^2&z_{211}-z_4\\
S_A^{(2)}&ae&z_{211}-z_{31}-z_{22}\\
S_A^{(2)}&cd&z_{22}-3z_{31}\\
S_A^{(2)}&cb&p-2z_{22}-z_{211}
\end{array}
\]
Thus the upper block gives
\begin{equation}
z_{211}=z_{31}+z_{22},\qquad p=2z_{22}+z_{211}.
\label{eq:upper-pentagram-weight-four}
\end{equation}
The lower block gives these equations and, in addition,
\[
z_{211}=z_4,\qquad z_{22}=3z_{31}.
\]
The second upper row is the last lower row minus the second lower row,
so the lower system visibly contains the upper one.

The resulting weight-four relations are
\begin{equation}
\begin{gathered}
\zeta(2,1,1)=\zeta(4),\qquad
\zeta(3,1)=\tfrac14\zeta(4),\qquad
\zeta(2,2)=\tfrac34\zeta(4),\\
\zeta(2)^2=2\zeta(2,2)+\zeta(4)=\tfrac52\zeta(4).
\end{gathered}
\label{eq:pentagram-weight-four}
\end{equation}

\subsection{Weight five}

Keep all eight convergent MZVs of weight five separate.
The corresponding coefficients of $A$ are
\[
\begin{array}{c|rrrrrrrr}
w&00001&00011&00101&00111&01001&01011&01101&01111\\\hline
A_w&-z_5&z_{41}&z_{32}&-z_{311}&z_{23}&-z_{221}&-z_{212}&z_{2111}
\end{array}
\]
where $0,1$ denote $e_0,e_1$. The other coefficients
ending in $1$ follow by shuffling a leading $1$, whose
regularized integral is zero.

For the upper block, initially keep
\[
p_{2,3}=\zeta(2)\zeta(3),\qquad
p_{2,21}=\zeta(2)\zeta(2,1)
\]
separate. The $3000$ relation multiples of residue
degree five have rank $2460$. The upper reflection
difference gives seven independent coefficient
equations:
\begin{equation}
\begin{gathered}
p_{2,3}=p_{2,21},\\
z_{41}=z_{311}=2z_{2111}-p_{2,3},\\
z_{32}=z_{221}=-\tfrac{11}{2}z_{2111}+3p_{2,3},\\
z_{23}=z_{212}=\tfrac92z_{2111}-2p_{2,3}.
\end{gathered}
\label{eq:upper-pentagram-weight-five}
\end{equation}
The coefficient $z_5$ does not occur in these equations.

For the lower block, use the already established
weight-three relation \eqref{eq:pentagram-weight-three}
and put $p_{23}=p_{2,3}=p_{2,21}$.
The mixed product from weights two and three contributes
\[
p_{23}\bigl(c(b^2+2bc-cb)+(a+b+c+d+e)cb\bigr).
\]
This block has residue degree three. The sixty cubic
relations obtained by multiplying each quadratic
relation on the left and on the right by each generator
are independent, leaving a quotient of dimension $65$.
The following seven coefficients of the lower
reflection difference span its equations:
\[
\begin{array}{c|l}
\text{residue word}&\text{coefficient in }S_A^{(2)}-\sigma S_A^{(2)}\\\hline
a^3&-z_5+z_{2111}\\
ba^2&-z_{41}+z_{311}\\
aea&z_{23}+z_{221}+z_{2111}-p_{23}\\
b^2c&z_{41}+z_{32}+z_{23}-z_{2111}\\
bc^2&z_{311}+z_{221}+z_{212}-z_{2111}\\
dc^2&5z_{311}+z_{221}-z_{212}\\
cda&-z_{41}-4z_{311}+z_{23}-z_{221}
\end{array}
\]
Solving gives
\begin{equation}
\begin{aligned}
\zeta(2,1,1,1)&=\zeta(5),\\
\zeta(4,1)=\zeta(3,1,1)
 &=2\zeta(5)-\zeta(2)\zeta(3),\\
\zeta(3,2)=\zeta(2,2,1)
 &=3\zeta(2)\zeta(3)-\tfrac{11}{2}\zeta(5),\\
\zeta(2,3)=\zeta(2,1,2)
 &=\tfrac92\zeta(5)-2\zeta(2)\zeta(3).
\end{aligned}
\label{eq:pentagram-weight-five}
\end{equation}
After identifying the two products in lower weight, the
lower block adds $z_{2111}=z_5$ to
\eqref{eq:upper-pentagram-weight-five}. These MZV reductions
follow from the coefficient equations; they were not used
to form them. No independence of $\zeta(5)$ and
$\zeta(2)\zeta(3)$ is assumed.

\subsection{A Maurer--Cartan presentation}
\label{app:maurer-cartan}

We consider the two reflection equations for arbitrary series
with constant term $1$ and zero linear terms. Let $R$ be the
graded rational residue algebra
\eqref{eq:five-residue-relations}, and put
\[
R_n^-=\{r\in R_n:\sigma(r)=-r\}.
\]
Define a weight-completed graded vector space by
\[
\mathfrak l^1=\prod_{n\ge2}\bigoplus_{|w|=n}\Q x_w,
\qquad
\mathfrak l^2=\prod_{n\ge2}(R_n^-\oplus R_{n-2}^-),
\qquad
\mathfrak l^j=0\quad(j\ne1,2).
\]
Here $x_w$ is labelled by a binary word, and the two entries
in the $n$-th factor of $\mathfrak l^2$ both have weight $n$.
For $w$ ending in $e_1$, use the expressions in
\eqref{eq:mzv-coefficient-extraction} to put
\[
p(x_w)=\binom{L_w^{(1)}}{L_w^{(2)}},\qquad q(x_w)=C_w;
\]
set both expressions to zero for words ending in $e_0$.
Extend $p,q$ linearly and continuously, and put $J=\binom10$.
For $\alpha,\beta\in\mathfrak l^1$, define
\[
\begin{aligned}
\partial\alpha&=(1-\sigma)\bigl(p(\alpha)+Jq(\alpha)\bigr),\\
[\alpha,\beta]&=(1-\sigma)\bigl(
p(\alpha)q(\beta)+p(\beta)q(\alpha)\bigr).
\end{aligned}
\]
Declare $\mathfrak l^2$ central and set
$\partial\mathfrak l^2=0$.

\begin{prop}\label{prop:pentagram-maurer-cartan}
These operations make $\mathfrak l$ a complete differential
graded Lie algebra. For every commutative $\Q$-algebra $B$,
write
\[
A=1+\sum_{|w|\ge2}A_ww,
\qquad
\alpha_A=\sum_{|w|\ge2}A_wx_w
 \in\mathfrak l^1\widehat\otimes B.
\]
Then
\[
\partial\alpha_A+\frac12[\alpha_A,\alpha_A]
 =S_A-\sigma(S_A).
\]
Thus the two pentagram reflection equations are precisely
the Maurer--Cartan equation in $\mathfrak l$.
\end{prop}

\begin{proof}
The bracket of degree-one elements is symmetric, as required
by graded skew-symmetry. Its values and those of the differential
are central in degree two. Hence the Jacobi and derivation
identities hold, and $\partial^2=0$. The differential preserves
weight and the bracket adds weights, so the operations are
continuous. Finally,
\[
S_A=(J+p(\alpha_A))(1+q(\alpha_A)).
\]
Since $\sigma(J)=J$, its reflection difference is
\[
(1-\sigma)\bigl(p(\alpha_A)+Jq(\alpha_A)
                         +p(\alpha_A)q(\alpha_A)\bigr),
\]
which is the displayed Maurer--Cartan expression.
\end{proof}

\begin{rem}
This presentation encodes the two reflection equations themselves;
pure-power normalization and reversal are the additional conditions
of Subsection~\ref{subsec:shuffle}. In particular, the generators
$x_w$ for words ending in $e_0$ are closed and central, since
these coefficients do not occur in the reflection equations.
\end{rem}

\bibliographystyle{alpha}
\bibliography{hodge.bib}
\end{document}